\documentclass[a4paper,12pt]{amsart}\pdfoutput=1 
\usepackage[margin=1.25in]{geometry}
\usepackage[utf8]{inputenc}
\usepackage{amssymb,amsmath,amsthm,mathtools,enumerate,epsfig,float,epigraph} 

\newenvironment{enui}{\begin{enumerate}[(i)]}{\end{enumerate}}
\newenvironment{enua}{\begin{enumerate}[(a)]}{\end{enumerate}}

\newcommand{\factor}{0.73}
\newcommand{\arxiv}[1]{#1}
\newcommand{\journal}[1]{}

\newcommand{\R}{\mathbb{R}}
\newcommand{\CC}{\mathcal{C}}
\newcommand{\CCC}{\mathcal{C}}
\newcommand{\N}{\mathbb{N}}
\newcommand{\NN}{\mathcal{N}}
\newcommand{\A}{\mathcal{A}}
\newcommand{\bd}{\partial}
\newcommand{\btop}[2]{\partial^{#2}#1} 
\newcommand{\F}{\mathcal{F}}
\newcommand{\G}{{\mathcal{G}}}
\newcommand{\wt}{\widetilde}
\newcommand{\U}{\mathcal{U}}
\newcommand{\Vol}{\operatorname{Vol}}
\newcommand{\al}{\alpha}
\newcommand{\be}{\beta}
\newcommand{\OMEGA}[2]{\Omega^{#1,#2}}
\newcommand{\OM}{\BAR{\Omega}}
\newcommand{\Om}{\Omega}
\newcommand{\om}{\omega}
\newcommand\OP[1]{\wt{\operatorname{Op}}_{#1}}
\newcommand\OEll[1]{\mathcal{O}^{\operatorname{Ell}}_{#1}} 
\newcommand\MEll[1]{\mathcal{M}^{\operatorname{Ell}}_{#1}} 
\newcommand{\Ell}[1]{Ell_{#1}} 
\newcommand{\ELL}[1]{\wt{Ell}_{#1}} 
\newcommand{\sh}{Sh} 
\newcommand{\nc}{\mathcal{NC}ap} 
\newcommand{\caps}{\mathcal{C}ap} 
\newcommand{\Caps}{\widetilde{\mathcal{C}ap}}
\newcommand{\p}{\mathcal{P}}
\newtheorem{thm}{Theorem}
\newtheorem*{Thm}{Theorem}
\newtheorem{cor}[thm]{Corollary}

\newtheorem{prop}[thm]{Proposition}
\newtheorem{pb}[thm]{Problem}
\newtheorem*{Pb}{Problem}
\newtheorem{q}[thm]{Question}

\newtheorem{lemma}[thm]{Lemma}
\newtheorem{rmk}[thm]{Remark}
\newtheorem{rmks}[thm]{Remarks}
\newtheorem*{Rmk}{Remark}
\newtheorem*{Rmks}{Remarks}
\newtheorem{defi}[thm]{Definition}
\newtheorem*{Defi}{Definition}
\newtheorem{xpl}[thm]{Example}

\newtheorem*{Xpls}{Examples}
\newtheorem{claim}{Claim}
\newcommand{\wo}{\setminus}
\newcommand\reff[1]{(\ref{#1})}
\newcommand\sub{\subseteq}
\newcommand\ev{\operatorname{ev}}
\newcommand\si{\sigma}
\newcommand\Si{\Sigma}
\newcommand\wrt{w.r.t.~}
\newcommand\PP{\mathcal{P}}
\newcommand\M{\mathcal{M}}
\newcommand\MM{\mathfrak{M}}
\newcommand\lan{\langle}
\newcommand\ran{\rangle}
\newcommand\B{\mathcal{B}}
\newcommand\x{\times}
\newcommand\nn{{\nonumber}}
\newcommand\OO{\mathcal{O}} 
\newcommand\oo{O} 

\newcommand{\BAR}[1]{{\overline{#1}}}
\newcommand\disj{\sqcup}
\newcommand\Disj{\bigsqcup}
\newcommand\Sum[1]{\sum\nolimits_{#1}}
\newcommand\sump{\sum\nolimits'}
\newcommand\cEH[1]{c^{\operatorname{EH}}_{#1}}
\newcommand\id{{\operatorname{id}}}
\newcommand\Int{\operatorname{Int}}
\newcommand\INT{\operatorname{int}}
\renewcommand\phi{\varphi}
\newcommand\then{\Rightarrow}
\newcommand\iso{\cong}
\newcommand\Sim{\approx}
\newcommand\pr{\operatorname{pr}}
\newcommand\im{\operatorname{im}}
\newcommand\cc{\wt c}
\newcommand\follows{\Leftarrow}
\newcommand\omst{\omega_{\operatorname{st}}}
\newcommand\conn[1]{\mathcal{C}_{#1}}
\newcommand\D{\mathcal{D}}
\newcommand\E{\mathcal{E}}
\newcommand\WLOG{W.l.o.g.~}
\newcommand{\hhat}[1]{\widehat{#1}}
\newcommand\SYMP[1]{\operatorname{Symp}^{#1}}
 
\title[Generating sets for symplectic capacities]{Generating sets and representability for symplectic capacities}
\author{Du\v{s}an Joksimovi\'c\and Fabian Ziltener}
\address{Affiliation of D.~Joksimovi\'c: Institut de Math\'ematiques de Jussieu - Paris Rive Gauche (IMJ-PRG)\\
4 place Jussieu\\
Boite Courrier 247\\
75252 Paris Cedex 5}
\email{dusanjoksimovic123@gmail.com}
\address{Affiliation of F.~Ziltener: Utrecht University\\
mathematics institute\\
Hans Freudenthalgebouw, office 606\\
Budapestlaan 6\\
3584 CD Utrecht\\
The Netherlands}
\email{f.ziltener@uu.nl}
\thanks{D.~Joksimovi\'c's work on this article was funded by NWO (The Netherlands Organisation for Scientific Research) via the TOP grant 614.001.508. F.~Ziltener worked on this article during a visit at FIM - the Institute for Mathematical Research at ETH Z\"urich. He gratefully acknowledges the hospitality of ETH}

\begin{document}

\begin{abstract}
K.~Cieliebak, H.~Hofer, J.~Latschev, and F.~Schlenk (CHLS) posed the problem of finding a minimal generating set for the (symplectic) capacities on a given symplectic category. We show that if the category contains a certain one-parameter family of objects, then every countably Borel-generating set of (normalized) capacities has cardinality (strictly) bigger than the continuum. This appears to be the first result regarding the problem of CHLS, except for two results of D.~McDuff about the category of ellipsoids in dimension 4.

We also prove that every finitely differentiably generating set of capacities on a given 
symplectic category is uncountable, provided that the category contains a one-parameter family of symplectic manifolds that is ``strictly volume-increasing'' and ``embedding-capacity-wise constant''. 
It follows that the Ekeland-Hofer capacities and the volume capacity do not finitely differentiably generate all generalized capacities on the category of ellipsoids. This answers a variant of a question of CHLS.

In addition, we prove that if a given symplectic category contains a certain one-parameter family of objects, then almost no normalized capacity is domain- or target-representable. This provides some solutions to two central problems of CHLS.
\end{abstract}

\maketitle

\tableofcontents

\epigraph{``The introduction of the cipher 0 or the group concept was general nonsense too, and mathematics was more or less stagnating for thousands of years because nobody was around to take such childish steps \ldots''}{Alexander Grothendieck}

\section{Introduction}
\subsection{The problems and special cases of the main results}\label{sec:special cases}
In this section we state special cases of the main results, since they seem easier to digest. We postpone the statements of the general results until the next section.\footnote{In the spirit of Grothendieck we have aimed at formulating our main results in the greatest possible generality.\footnotemark~The anonymous referee for this article was afraid that this would scare off readers. For this reason we first state special cases of the results.} 
\footnotetext{Probably we failed miserably in this attempt.}

Our first main result is concerned with generating sets of capacities on symplectic categories. To explain a special case of this result, let $n\in\N_0:=\{0,1,\ldots\}$. We denote by $\SYMP{2n}$ the category of all symplectic manifolds\footnote{In this article ``manifold'' refers to a smooth ($C^\infty$) real finite-dimensional manifold. It is allowed to be disconnected and have boundary.} %
of dimension $2n$, with morphisms given by the symplectic embeddings\footnote{By an embedding we mean an injective smooth immersion with continuous inverse. We do not impose any condition involving the boundaries of the two manifolds.}. 
Recall that a subcategory $\CC'$ of a category $\CC$ is called \emph{isomorphism-closed (in $\CC$)} iff every isomorphism of $\CC$ starting at some object of $\CC'$ is a morphism of $\CC'$.\footnote{In particular, it ends at some object of $\CC'$.}
\begin{Defi}[(weak) symplectic category] A \emph{weak symplectic category in dimension $2n$} is a subcategory $\CCC=(\OO,\M)$ \footnote{Here $\OO$ and $\M$ denote the classes of objects and morphisms of $\CCC$, respectively.} %
of $\SYMP{2n}$, such that for every $(M,\om)\in\OO$ and every $a\in(0,\infty)$ we have $(M,a\om)\in\OO$. 
We call such a $\CCC$ a \emph{symplectic category} if it is isomorphism-closed.  
\end{Defi}
\begin{Rmk}\textnormal{Weak symplectic categories were first defined in \cite[2.1.~Definition, p.~5]{CHLS}, where the adjective ``weak'' is left out. We explain the reason for our change in terminology in Remark \ref{rmk:iso closed} on p.~\pageref{rmk:iso closed}.
}
\end{Rmk}
We denote by $B^m_r$ ($\BAR B^m_r$) the open (closed) ball of radius $r$ around 0 in $\R^m$, and abbreviate
\[B:=B^{2n}_1,\quad Z^{2n}_r:=B^2_r\x\R^{2n-2},\quad Z:=Z^{2n}_1.\]
We equip $B^{2n}_r$ and $Z^{2n}_r$ with the standard symplectic form $\omst$. Let $\CCC=(\OO,\M)$ be a symplectic category.
\begin{defi}\label{defi:cap }
A \emph{generalized capacity on $\CCC$} is a function
\[c:\OO\to[0,\infty]\]
with the following properties:
\begin{enui}
\item\label{defi:cap :mon}\textbf{(monotonicity)} If $(M,\om)$ and $(M',\om')$ are two objects in $\OO$ between which there exists a $\CCC$-morphism, then
\[c(M,\om)\leq c(M',\om').\]
\item\label{defi:cap :conf}\textbf{(conformality)} For every $(M,\om)\in\OO$ and $a\in(0,\infty)$ we have 
\[c(M,a\om)=a\,c(M,\om).\]
\end{enui}
Assume now that $B,Z\in\OO$. Let $c$ be a generalized capacity on $\CCC$. We call $c$ a \emph{capacity} iff is satisfies:
\begin{enui}
\item[(iii)]\label{defi:cap :non-triv}\textbf{(non-triviality)} $c(B)>0$ and $c(Z)<\infty$.\end{enui}
We call it \emph{normalized} iff it satisfies:
\begin{enui}
\item[(iv)]\label{defi:cap :norm} \textbf{(normalization)} $c(B)=c(Z)=\pi$.\footnote{In \cite[2.1.~Definition, p.~5]{CHLS} only the condition $c(B)=1$ is imposed here. The second part of our first main result, Theorem \ref{thm:card cap gen} below, 
holds even with our stronger definition.}
\end{enui}
We denote
\[\caps(\CCC):=\big\{\textrm{generalized capacity on }\CCC\big\}.\]
\end{defi}
\begin{Rmk}\textnormal{There is a set-theoretic issue with this definition, which we will resolve in Definition \ref{defi:cap} in the next section. 
Compare to Remark \ref{rmk:iso closed}.}
\end{Rmk}
\begin{xpl}[embedding capacities]\label{xpl:emb cap }\textnormal{Let $\CCC=(\OO,\M)$ be a symplectic category in dimension $2n$ and $(M,\om)$ an object of $\SYMP{2n}$. We define the \emph{domain-embedding capacity for $(M,\om)$ on $\CCC$} to be the function
\begin{eqnarray*}&c_{M,\om}:=c^\CCC_{M,\om}:\OO\to[0,\infty],&\\
&c_{M,\om}(M',\om'):=\sup\left\lbrace c\in(0,\infty)\,\big|\,\exists\,\textrm{ symplectic embedding }(M,c\om)\to(M',\om')\right\rbrace.&
\end{eqnarray*}
We define the \emph{target-embedding capacity for $(M,\om)$ on $\CCC$} to be the function
\begin{eqnarray*}&c^{M,\om}:=c^{M,\om}_\CCC:\OO\to[0,\infty],&\\
&c^{M,\om}(M',\om'):=\inf\left\lbrace c\in(0,\infty)\,\big|\,\exists\,\textrm{ symplectic embedding }(M',\om')\to(M,c\om)\right\rbrace.&
\end{eqnarray*}
These are generalized capacities.%
\footnote{\label{emb cap}In the definition on p.~13 in \cite{CHLS} $(M,\om)$ is assumed to be an object of $\CCC$, and the morphisms in the definitions of the embedding capacities are asked to be $\CCC$-morphisms. However, in \cite[Example 2, p.~14]{CHLS} the definition is applied with an $(M,\om)$ that is not an object of $\CCC$. In order to make that example work, one needs to allow for $\SYMP{2n}$-morphisms in the definition of the embedding capacities.
} 
We define the \emph{Gromov width on $\CCC$} to be
\begin{equation}\label{eq:w}w:=\pi c^\CCC_{B,\omst}.
\end{equation}
If $B,Z\in\OO$, then by Gromov's nonsqueezing theorem the Gromov width is a normalized capacity.
}
\end{xpl}
Capacities on the category of all symplectic manifolds of a fixed dimension were introduced by I.~Ekeland and H.~Hofer in \cite{EH1,EH2}. They measure how much a given symplectic manifold does not embed into another one. In \cite{CHLS} the notion of a symplectic capacity on a general weak symplectic category was introduced.\footnote{For the set-theoretic reason mentioned in Remark \ref{rmk:iso closed} one needs to ask that the category is isomorphism-closed or small, in order to make sense of this definition.} 
For an overview over symplectic capacities we refer to \cite{CHLS,Schlenk2018,SchlenkEmbedding} and references therein.\\

To define the notion of a CHLS capacity-generating set, let $f:[0,\infty]^\ell\to[0,\infty]^{\ell'}$ be a map. We call it \emph{homogeneous} iff it is positively 1-homogoneous, i.e., iff $f(ax)=af(x)$ for every $a\in(0,\infty)$ and $x\in[0,\infty]^\ell$.\footnote{Here we use the convention $a\cdot\infty:=\infty$ for every $a\in(0,\infty)$.} 
We equip $[0,\infty]^\ell$ with the partial order given by $x\leq y$ iff $x_i\leq y_i$ for every $i\in\{1,\ldots,\ell\}$. We call $f$ \emph{monotone} iff it preserves this partial order. As pointed out by the authors of \cite{CHLS} in that article, if $\ell'=1$, $f$ is homogeneous and monotone, and $c_1,\ldots,c_\ell$ are generalized capacities, then $f\circ(c_1,\ldots,c_\ell)$ is again a generalized capacity. Homogeneity and monotonicity are preserved under compositions.
\begin{Xpls}\textnormal{The following functions are homogoneous and monotone:
\begin{itemize}
\item maximum, minimum
\item For every $a\in[0,\infty)^\ell$ and $p\in\R\wo\{0\}$ the function
\[f_{a,p}(x):=\sqrt[p]{\sum_{i=1}^\ell a_ix_i^p}.\,\footnote{Here we use the conventions $\infty+a=\infty$ for every $a\in[0,\infty]$, $\infty^p=\infty$ for every $p>0$, and $0^p:=\infty$ and $\infty^p:=0$ for every $p<0$.}
\]
In the case $a=\left(\frac1\ell,\ldots,\frac1\ell\right)$, $p=1$ the function $f_{a,p}$ is the arithmetic mean, and in the case $a=\left(\frac1\ell,\ldots,\frac1\ell\right)$, $p=-1$ it is the harmonic mean.
\item For every $p\in[0,\infty)^\ell$ satisfying $\sum_{i=1}^\ell p_i=1$ the function
\[x\mapsto\prod_{i=1}^\ell x_i^{p_i}.\]
In the case $p=\left(\frac1\ell,\ldots,\frac1\ell\right)$ this is the geometric mean.
\end{itemize}
}
\end{Xpls}
Let $\CCC$ be a symplectic category and $\G$ a subset of $\caps(\CCC)$. By a \emph{finite homogeneous monotone combination of $\G$} we mean an expression of the form $f\circ(c_1,\ldots,c_\ell)$, where $\ell\in\N_0$, \mbox{$f:[0,\infty]^\ell\to[0,\infty]$} is homogeneous and monotone, and $c_1,\ldots,c_\ell\in\G$. We define the \label{CHLS gen set}\emph{set CHLS generated by $\G$} to be the set of all functions $c:\OO\to[0,\infty]$ that are the pointwise limit of a sequence of finite homogeneous monotone combinations of $\G$. Since pointwise limits preserve homogeneity and monotonicity, the set CHLS generated by $\G$ consists again of generalized capacities.

\label{CHLS cap gen} By a \emph{CHLS (generalized-)(capacity-)generating set (for $\CCC$)} we mean a subset $\G$ of $\caps(\CCC)$, whose CHLS generated set equals $\caps(\CCC)$.\footnote{This notion is introduced in \cite[Problem 5, p.~17]{CHLS}, where it is just called a ``generating system''. (The authors of that article do not explicitly state that $\G$ should be a subset of $\caps(\CCC)$, but presumably they implicitly ask for this.)} %
The set CHLS generated by $\G$ is obtained by combining capacities in a lot of ways. One may therefore expect that few capacities suffice to generate all the other capacities. It is tempting to even look for a generating set of capacities that is minimal, in the sense that none of its subsets is generating. This problem was posed by K.~Cieliebak, H.~Hofer, J.~Latschev, and F.~Schlenk (CHLS):

\begin{pb}[\cite{CHLS}, Problem 5, p.~17]\label{pb:gen} For a given symplectic category find a minimal CHLS capacity-generating set.
\end{pb}

In \cite{CHLS} CHLS also propose more restrictive notions of generating set, for example one in which the only allowed combining functions $f:[0,\infty]^\ell\to[0,\infty]$ are the minimum and the maximum. We call such a set \label{lim min max cap gen}\emph{limit-min/max \mbox{(capacity-)}generating}. A concrete instance of Problem \ref{pb:gen} for this variant of generation is the following.

\begin{q}\label{q:count gen} For a given symplectic category, does there exist a countable\footnote{By this we mean finite or countably infinite.} (minimal) limit-min/max capacity-generating set?
\end{q}

To our knowledge, up to now, Problem \ref{pb:gen} and Question \ref{q:count gen} have been completely open, except for the following two results of D.~McDuff.\footnote{There are of course some trivial cases in which Question \ref{q:count gen} is easy, e.g.~the case in which there are only finitely many $\CCC$-isomorphism classes.} 
The first result states that the ECH-capacities are generating in a weaker sense for the category of ellipsoids in dimension 4, see Theorem \ref{thm:gen ell} on p.~\pageref{thm:gen ell}. The second result states that the Ekeland-Hofer capacities together with the volume capacity are not CHLS generating for the restriction category of (open) ellipsoids (as defined on p.~\pageref{restr cat ell}) in dimension 4. (See \cite[Corollary 1.4]{McDEll}.)

By the \emph{continuum} we mean the cardinality of the set of real numbers. The following special case of our first main result (Theorem \ref{thm:card cap gen} on p.~\pageref{thm:card cap gen}) provides a negative answer to Question \ref{q:count gen}. In fact, it implies that there is not even a generating set of cardinality at most that of the continuum.
\begin{thm}[cardinality of generating set]\label{thm:lim min max}If $n\geq2$ then every limit-min/max capacity-generating set for $\SYMP{2n}$ has cardinality (strictly) bigger than the continuum.
\end{thm}
This result follows immediately from Corollary \ref{cor:gen}\reff{cor:gen:caps} on p.~\pageref{cor:gen}, which follows from Theorem \ref{thm:card cap gen}. This corollary greatly generalizes Theorem \ref{thm:lim min max} by extending it to (differential) form categories containing a certain one-parameter family of objects and by replacing the notion of limit-min/max generation by the much weaker notion of countable Borel-generation. (See Definition \ref{defi:gen set} on p.~\pageref{defi:gen set} and Remarks \ref{rmk:Borel} and \ref{rmk:gen}.)
\footnote{We also formulate a version of the result for normalized capacities. See Corollary \ref{cor:gen}\reff{cor:gen:ncaps}.} 
The idea of proof of part of Theorem \ref{thm:card cap gen} is to use Stokes' theorem for helicity. (We will explain this in Subsection \ref{subsec:ideas}.)

Theorem \ref{thm:lim min max} and - more generally - Corollary \ref{cor:gen} diminish the hope of finding manageable generating sets of (generalized) symplectic capacities.\\

Our second main result states that every suitably generating set for the generalized capacities on a small weak form category is uncountable, if the category contains a one-parameter family of objects that is ``strictly volume-increasing'' and ``embedding-capacity-wise constant''. 
Morally speaking, this hypothesis is weaker than those of the first main result. As a special case the second main result holds for the category of ellipsoids in dimension at least 4. This answers a variant of a question of CHLS about the Ekeland-Hofer capacities. 

To explain this, let $(V,\om)$ be a (finite-dimensional) symplectic vector space. We abbreviate $V:=(V,\om)$ and define $\OP{V}$ to be the category with objects given by all pairs $(U,\om|U)$, where $U$ ranges over all open subsets of $V$, and morphisms between two objects $U,U'$ given by the restrictions $\phi|U$, where $\phi$ ranges over all symplectomorphisms of $V$, such that $\phi(U)\sub U'$. This is a small\footnote{A category is called small iff its objects and its morphisms form sets.} %
weak symplectic category. For such a category we define the notion of a generalized capacity as in Definition \ref{defi:cap }.

Let $i\in\N_0$. The $i$-th Ekeland-Hofer capacity $\cEH{i}$ is a capacity on $\OP{V}$, which is defined as a certain min-max involving the symplectic action, see \cite{EH1,EH2} or \cite[p.~7]{CHLS}. The capacity $\cEH{1}$ is normalized; the other Ekeland-Hofer capacities are not normalized.

The Ekeland-Hofer capacities are hard to compute. Their values are known for ellipsoids and polydisks, see \cite[Proposition 4, p. 562]{EH2} and \cite[Proposition 5, p. 563]{EH2}.

Recall that a (bounded, open, full) ellipsoid in $V$ is a set of the form $p^{-1}((-\infty,0))$, where $p:V\to\R$ is a quadratic polynomial function whose second order part is positive definite. We equip each ellipsoid $E$ with the restriction of $\om$ to $E$. We define the \label{restr cat ell}\emph{restriction category of ellipsoids} to be the full subcategory \label{ELL V}$\ELL{V}$ of $\OP{V}$ consisting of ellipsoids.\footnote{The word ``restriction'' refers to the fact that the morphisms of $\ELL{V}$ are restrictions of global symplectomorphisms. $\ELL{V}$ is a subcategory of the category $\Ell{V}$ whose objects and morphisms are given by ellipsoids and \emph{all} symplectic embeddings. We reserve the nicer notation $\Ell{V}$ for the bigger category.} 
The objects of $\ELL{V}$ are uniquely determined by the Ekeland-Hofer capacities, up to isomorphism, see \cite[FACT 10, p. 27]{CHLS}. Therefore the following question seems natural:
\begin{q}[\cite{CHLS}, Problem 15, p.~29]\label{q:EH} Do the Ekeland-Hofer capacities together with the volume capacity%
\footnote{One needs to include the volume capacity, since the Ekeland-Hofer capacities do not CHLS generate this capacity, see \cite[Example 10, p.~28]{CHLS}.} 
form a CHLS generating set for $\ELL{V}$?%
\footnote{In \cite{CHLS} it is only asked whether these generalized capacities generate all capacities, not all \emph{generalized} capacities. However, from the discussion that precedes the question it is clear that the authors intended to include the word \emph{generalized} here.} 
\end{q}
In the case $\dim V=4$ this question was answered negatively by D.~McDuff, see \cite[Corollary 1.4]{McDEll}. A special case of our second main result, Theorem \ref{thm:ell} below, answers Question \ref{q:EH} in the negative in dimension at least 4, if we replace ``CHLS generating'' by ``finitely differentiably generating'', as defined below. In fact, it states that every finitely differentiably generating set on the category of ellipsoids is uncountable.

Let $\CCC$ be a small weak symplectic category. We call a subset $\G$ of $\caps(\CCC)$ \emph{finitely differentiably \mbox{(capacity-)}generating for $\CCC$} iff for every $c\in\caps(\CCC)$ there exist $\ell\in\N_0$, $c_1,\ldots,c_\ell\in\G$, and a differentiable function $f:[0,\infty]^\ell\to[0,\infty]$, such that $c=f\big(c_1,\ldots,c_\ell\big)$.\footnote{Here we view $[0,\infty]$ as a compact 1-dimensional manifold with boundary. Its Cartesian power is a manifold with boundary and corners. The function $f$ is only assumed to be differentiable one time, with possibly discontinuous derivative.} 
A special case of our second main result is the following.
\begin{thm}[uncountability of every generating set for ellipsoids]\label{thm:ell} Let $V$ be a symplectic vector space of dimension at least 4. Then every finitely differentiably capacity-generating set for $\ELL{V}$ is uncountable.
\end{thm}
\begin{proof} This follows from Theorem \ref{thm:gen uncount} on p.~\pageref{eq:c Ma Ma'} and Example \ref{xpl:increasing} by considering the ellipsoids
\[M_a:=\left\{x=\big(x_1,\ldots,x_n\big)\in\R^{2n}=(\R^2)^n\,\bigg|\,\sum_{i=1}^{n-1}\Vert x_i\Vert^2+\frac{\Vert x_n\Vert^2}a<1\right\},\]
for $a\in A:=[1,\infty)$. Here $\Vert\cdot\Vert$ denotes the Euclidean norm on $\R^2$. Our hypothesis $n\geq2$ guarantees that the inequality ``$\leq$'' in condition \reff{eq:c Ma Ma'} on p.~\pageref{eq:c Ma Ma'} holds.
\end{proof}
The idea of proof of Theorem \ref{thm:gen uncount} is to use Lebesgue's Monotone Differentiation Theorem. (We explain this in Subsection \ref{subsec:ideas}.)

By Theorem \ref{thm:ell}, in particular, the Ekeland-Hofer capacities together with the volume capacity do not finitely differentiably generate the set of all generalized capacities on $\ELL{V}$. This provides a negative answer to the variant of Question \ref{q:EH} involving the notion of finite differentiable generation. We will generalize Theorem \ref{thm:ell} to small weak form categories containing a certain type of one-parameter family of objects. (See Theorem \ref{thm:gen uncount} on p.~\pageref{thm:gen uncount}.)\\

Our first main result (Theorem \ref{thm:card cap gen} below) also has immediate applications to two questions that CHLS prominently posed as Problems 1 and 2 in their article \cite{CHLS}. To explain these problems, let $n\in\N$ and $\CCC=(\OO,\M)$ be a symplectic category in dimension $2n$. 
\begin{Defi}[representability] Let $c$ be a capacity on $\CCC$. We call $c$ \emph{(symplectically) domain-/ target-representable} iff there exists a symplectic manifold $(M,\om)$, for which $c=c_{M,\om}$/ $c=c^{M,\om}$ (defined as in Example \ref{xpl:emb cap }). We call it \emph{connectedly target-representable} iff there exists a \emph{connected} symplectic manifold $(M,\om)$, for which $c=c^{M,\om}$.
\end{Defi}
\begin{Rmk}\textnormal{By definition, the topology of a manifold is second countable. Without this condition every capacity would be target-representable, if all objects of $\CCC$ are connected, see \cite[Example 2, p.~14]{CHLS}.
}
\end{Rmk}
\begin{q}[target-representability, \cite{CHLS}, p.~14, Problem 1]\label{q:rep to} Which (generalized) capacities on $\CCC$ are connectedly target-representable?%
\footnote{\label{target rep}In \cite{CHLS} this question is asked, based on the definition of the embedding capacities on p.~13 in that article. In that definition the embedding capacity is only defined for objects of $\CCC$. However, in \cite[Example 2, p.~14]{CHLS} the authors use the definition with a symplectic manifold that is not an object of $\CCC$. This suggests that CHLS are interested in Question \ref{q:rep to} with the modified definition given in Example \ref{xpl:emb cap }. Compare to footnote \ref{emb cap}.
}
\end{q}
\begin{q}[domain-representability, \cite{CHLS}, p.~14, Problem 2]\label{q:rep from} Which (generalized) capacities on $\CCC$ are domain-representable?%
\footnote{A remark similar to footnote \ref{target rep} applies.
}
\end{q}
In particular, one may wonder about the following:
\begin{q}\label{q:rep every} Is \emph{every} generalized capacity connectedly target-representable?
\end{q}
If the answer to this question is ``yes'', then this simplifies the study of capacities, since we may then identify every capacity with some symplectic manifold that target-represents it.

Apart from some elementary examples, up to now, the answers to Questions \ref{q:rep to},\ref{q:rep from}, and \ref{q:rep every} appear to be completely unknown. In order to answer Question \ref{q:rep every} negatively, it seems that we need to understand all symplectic embeddings from objects of $\CCC$ to all connected symplectic manifolds. At first glance this looks like a hopeless enterprise.

The following application of the first main result may therefore come as a surprise. Namely, the answer to Question \ref{q:rep every} is ``no'', if the symplectic category is of dimension at least 4 and contains a certain one-parameter family of objects. In fact, the answer remains ``no'', even if we ask the question only for \emph{normalized} capacities and drop the word ``connectedly''. Perhaps all the more unexpectedly, \emph{almost no} normalized capacity is target-representable. \label{almost no}Here we say that \emph{almost no} element of a given infinite set has a given property iff the subset of all elements with this property has smaller cardinality than the whole set. Similarly, almost no normalized capacity is domain-representable. The following application of the first main result concerns the special case of these statements for the whole symplectic category.
\begin{Thm}[representability] For every $n\geq2$ almost no normalized capacity on $\SYMP{2n}$ is domain- or target-representable.
\end{Thm}
This is an immediate consequence of Corollary \ref{cor:rep} on p.~\pageref{cor:rep}.\footnote{In that corollary we use a definition of representability that allows for the representing pair $(M,\om)$ to be a general manifold with a two-form.} %
It follows that there are as many normalized capacities on $\SYMP{2n}$ that are neither domain- nor target-representable, as there are normalized capacities overall (namely $\PP(\R)$-many\footnote{Here $\PP(S)$ denotes the power set of a set $S$.}
). This provides some answers to Questions \ref{q:rep to},\ref{q:rep from}, and \ref{q:rep every}.
\subsection{Organization of this article}
In Section \ref{sec:main} we state the main results in the general setting of (weak) differential form categories and deduce the applications about generating sets and representability. We also present the ideas of proof of the main results. Furthermore, we discuss the related result of D.~McDuff about monotone capacity-generation for the category of ellipsoids in dimension 4, and a potential application of our proof technique to morphism detection.

In Section \ref{proof:thm:card cap gen:caps} we formulate Theorem \ref{thm:card cap hel}, which generalizes part of the first main result (Theorem \ref{thm:card cap gen}(\ref{thm:card cap gen:caps},\ref{thm:card cap gen:ncaps})). It states that the cardinality of the set of (normalized) capacities equals that of $\PP(\R)$ for every form category containing each disjoint union $M_a\disj M_{-a}$ for a suitable one-parameter family of manifolds with forms $(M_a,\om_a)_{a\in A_0}$. This family needs to satisfy the following crucial condition. We denote by $I_a$ the set of connected components of the boundary of $M_a$, and $I:=(I_a)_{a\in A_0}$. Then the collection of boundary helicities associated with $(M_a,\om_a)_{a\in A_0}$ is an $I$-collection. We introduce the notions of helicity and of an $I$-collection in this section. We also state Proposition \ref{prop:I collection}, which provides sufficient criteria for the helicity hypothesis of Theorem \ref{thm:card cap hel}.

In Sections \ref{sec:proof:thm:card cap hel} \arxiv{and \ref{sec:proof:prop:I collection}} we prove Theorem \ref{thm:card cap hel} \arxiv{and Proposition \ref{prop:I collection}}.

Section \ref{sec:proof:thm:card cap gen:gen} contains the proof of the last part of the first main result (Theorem \ref{thm:card cap gen}\reff{thm:card cap gen:gen}), which states that every set of cardinality at most that of $\R$ countably Borel-generates a set of cardinality at most that of $\R$.

Section \ref{proof:thm:gen uncount} is devoted to the proof of our second main result, Theorem \ref{thm:gen uncount}, stating that every finitely differentiably generating set of capacities is uncountable if the category contains a certain type of one-parameter family of objects.

In Section \ref{sec:form} we prove an auxiliary result, which states that the set of diffeomorphism classes of manifolds has cardinality that of $\R$. We also show that the same holds for the set of all equivalence classes of pairs $(M,\om)$, where $M$ is a manifold and $\om$ a differential form on $M$.

Finally, in Section \ref{sec:gen ell} we deduce Theorem \ref{thm:gen ell} (monotone generation for ellipsoids) from McDuff's characterization of the existence of symplectic embeddings between ellipsoids.
\subsection{Acknowledgments}
The authors would like to thank the anonymous referee for suggesting to make the introduction more readable by mentioning only special cases of the main results. We are grateful to Felix Schlenk for sharing his expertise on symplectic capacities. We also thank Urs Frauenfelder for an interesting discussion, Timofey Rodionov and Ji\v{r}\'i Spurn\'y for their help with a question about Baire and Borel hierarchies, and Asaf Karagila for his help with a question about aleph and bet numbers.
\section{Main results and applications}\label{sec:main}
In this section we state the main results of this article in the general setting and deduce the applications about generating sets and representability, taking care of the set-theoretic issue that was mentioned in the previous section. 
\subsection{Cardinalities of the set of capacities and of the generated set} 
The first main result provides conditions under which every generating set of capacities on a differential form category is bigger than the continuum. To define the notion of such a category, let $m,k\in\N_0:=\{0,1,\ldots\}$. We define $\OMEGA{m}{k}$ to be the following category:
\begin{itemize}
\item Its objects are pairs $(M,\om)$, where $M$ is a manifold of dimension $m$, and $\om$ is a differential $k$-form on $M$.
\item Its morphisms are embeddings that intertwine the differential forms.
\end{itemize}
\begin{Defi} A \emph{weak $(m,k)$-(differential) form category} is a subcategory $\CCC=(\OO,\M)$ of $\OMEGA{m}{k}$, such that if $(M,\om)\in\OO$ and $a\in(0,\infty)$ then $(M,a\om)\in\OO$. We call such a $\CCC$ a \emph{$(m,k)$-form category} iff it is also isomorphism-closed. 

A \emph{(weak) symplectic category (in dimension $2n$)} is a (weak) $(2n,2)$-form category whose objects are symplectic manifolds.
\end{Defi}

\begin{Xpls}[(weak) $(m,k)$-form category]\textnormal{
\begin{enui}\item Let $\MM$ be a diffeomorphism class of smooth manifolds of dimension $m$. The full subcategory of $\OMEGA{m}{k}$ whose objects $(M,\om)$ satisfy $M\in\MM$, is an $(m,k)$-form category.
\item Let $(M,\om)$ be an object of $\OMEGA{m}{k}$. Consider the category with objects given by all pairs $(U,\om|U)$, where $U$ ranges over all open subsets of $M$, and morphisms between two objects $U,U'$ given by the restrictions $\phi|U$, where $\phi$ ranges over all isomorphisms of $(M,\om)$, such that $\phi(U)\sub U'$. This is a small weak $(m,k)$-form category, which is not isomorphism-closed, hence not an $(m,k)$-form category.
\end{enui}
}
\end{Xpls}
\begin{rmk}[isomorphism-closedness]\label{rmk:iso closed}
\textnormal{Symplectic categories were first defined in \cite[2.1.~Definition, p.~5]{CHLS}. In that definition isomorphism-closedness is not assumed. However, this condition is needed in order to avoid the following set-theoretic issue in the definition of the notion of a symplectic capacity on a given symplectic category $\CCC$.} 

\textnormal{This article is based on ZFC, the Zermelo-Fraenkel axiomatic system together with the axiom of choice. A category is a pair consisting of classes of objects and morphisms. Formally, in ZFC there is no notion of a ``class'' that is not a set. The system \emph{can} handle a ``class'' that is determined by a wellformed formula, such as the ``class'' of all sets or the ``class'' of all symplectic manifolds, by rewriting every statement involving the ``class'' as a statement involving the formula.} 

\textnormal{However, it is not possible in ZFC to define the ``class'' of all maps between two classes, even if the target class is a set. In particular, it is a priori not possible to define the ``class'' of all symplectic capacities on a given symplectic category. Our assumption that $\CCC$ is isomorphism-closed makes it possible to define this ``class'' even as a set, see below.}
\end{rmk}
We now define the notion of a (generalized) capacity on a given form category. Let $S$ be a set. By $|S|$ we denote the \emph{(von Neumann) cardinality of $S$}, i.e., the smallest (von Neumann) ordinal that is in bijection with $S$. For every pair of sets $S,S'$ we denote by ${S'}^S$ the set of maps from $S$ to $S'$. For every pair of cardinals $\al,\be$ \footnote{i.e., cardinalities of some sets} we also use $\be^\al$ to denote the cardinality of $\be^\al$. Recursively, we define $\beth_0:=\N_0$, and for every $i\in\N_0$, the cardinal $\beth_{i+1}:=2^{\beth_i}$.\footnote{$\beth$ (bet) is the second letter of the Hebrew alphabet.}

Let $\CCC=(\OO,\M)$ be an $(m,k)$-form category. We define the set
\begin{equation}\label{eq:OO0}\OO_0:=\big\{(M,\om)\in\OO\,\big|\,\textrm{The set underlying $M$ is a subset of }\beth_1.\big\}.\end{equation}
\begin{defi}\label{defi:cap} A \emph{generalized capacity on $\CCC$} is a function
\[c:\OO_0\to[0,\infty]\]
with the following properties:
\begin{enui}
\item\label{defi:cap:mon}\textbf{(monotonicity)} If $(M,\om)$ and $(M',\om')$ are two objects in $\OO_0$ between which there exists a $\CCC$-morphism, then
\[c(M,\om)\leq c(M',\om').\]
\item\label{defi:cap:conf}\textbf{(conformality)} For every $(M,\om)\in\OO_0$ and $a\in(0,\infty)$ we have 
\[c(M,a\om)=a\,c(M,\om).\]
\end{enui}
Assume now that $k=2$, $m=2n$ for some integer $n$, and that $\OO_0$ contains some objects $B_0,Z_0$ that are isomorphic to $B,Z$ (the open unit ball and cylinder). Let $c$ be a generalized capacity on $\CCC$. We call $c$ a \emph{capacity} iff is satisfies:
\begin{enui}
\item[(iii)]\label{defi:cap:non-triv}\textbf{(non-triviality)} $c(B_0)>0$ and $c(Z_0)<\infty$.\footnote{These conditions do not depend on the choices of $B_0,Z_0$, since $c$ is is invariant under isomorphisms by monotonicity.}
\end{enui}
We call it \emph{normalized} iff it satisfies:
\begin{enui}
\item[(iv)]\label{defi:cap:norm} \textbf{(normalization)} $c(B_0)=c(Z_0)=\pi$.
\end{enui}
We denote by
\[\caps(\CCC),\quad\nc(\CCC)\]
the sets of generalized and normalized capacities on $\CCC$. If $\CCC$ is a symplectic category then we call a (generalized/ normalized) capacity on $\CCC$ also a \emph{(generalized/ normalized) symplectic capacity}.
\end{defi}
\begin{xpl}[embedding capacities]\label{xpl:emb cap}
\textnormal{Let $\CCC=(\OO,\M)$ be an $(m,k)$-form category and $(M,\om)$ an object of $\Om^{m,k}$. We define $\OO_0$ as in \eqref{eq:OO0} and the \emph{domain-embedding capacity for $(M,\om)$ on $\CCC$} to be the function
\begin{eqnarray}\nn&c_{M,\om}:=c^\CCC_{M,\om}:\OO_0\to[0,\infty],&\\
\label{eq:c M om}&c_{M,\om}(M',\om'):=\sup\left\lbrace c\in(0,\infty)\,\big|\,\exists\,\Om^{m,k}\textrm{-morphism }(M,c\om)\to(M',\om')\right\rbrace.&
\end{eqnarray}
We define the \emph{target-embedding capacity for $(M,\om)$ on $\CCC$} to be the function
\begin{eqnarray*}&c^{M,\om}:=c^{M,\om}_\CCC:\OO_0\to[0,\infty],&\\
&c^{M,\om}(M',\om'):=\inf\left\lbrace c\in(0,\infty)\,\big|\,\exists\,\Om^{m,k}\textrm{-morphism }(M',\om')\to(M,c\om)\right\rbrace.&
\end{eqnarray*}
These are generalized capacities.
}
\end{xpl}
\begin{rmks}[set of capacities and isomorphism-closedness]\label{rmk:cap}\begin{enui}
\item\textnormal{The collections $\caps(\CCC)$ and $\nc(\CCC)$ are indeed sets, since $\OO_0$ is a set.}
\item\textnormal{Heuristically, let us denote by $\Caps(\CCC)$ the ``subclass'' of ``$[0,\infty]^\OO$'' consisting of all ``functions'' satisfying conditions the monotonicity and conformality conditions of Definition \ref{defi:cap}. Formally, the restriction from $\OO$ to $\OO_0$ induces a bijection between $\Caps(\CCC)$ and $\caps(\CCC)$.%
\footnote{This follows from the fact that every object of $\OMEGA{m}{k}$ is isomorphic to one whose underlying set is a subset of $\beth_1$, and the assumption that $\CCC$ is isomorphism-closed. To prove the fact, recall that by definition, the topology of every manifold $M$ is second countable. Using the axiom of choice, it follows that its underlying set has cardinality $\leq\beth_1$. This means that there exists an injective map $f:M\to\beth_1$. Consider now an object $(M,\om)$ of $\OMEGA{m}{k}$. Pushing forward the manifold structure and $\om$ via a map $f$, we obtain an object of $\OMEGA{m}{k}$ isomorphic to $(M,\om)$, whose underlying set is a subset of $\beth_1$. This proves the fact.} 
This means that our definition of a generalized capacity corresponds to the intuition behind the usual ``definition''. Here we use isomorphism-closedness of $\CCC$. Compare to Remark \ref{rmk:iso closed}.}
\item\label{rmk:cap:OO0}\textnormal{Isomorphism-closedness of $\CCC$ implies that there is a canonical bijection between $\caps(\CCC)$ and the set of generalized capacities that we obtain by replacing $\OO_0$ by any sub\emph{set} of $\OO$ that contains an isomorphic copy of each element of $\OO$. Such a subset can for example be obtained by replacing $\beth_1$ in \reff{eq:OO0} by any set of cardinality at least $\beth_1$.\footnote{This follows from an argument as in the last footnote.} This means that our definition of $\caps(\CCC)$ is natural.}
\item\textnormal{Example \ref{xpl:emb cap} generalizes Example \ref{xpl:emb cap }, taking care of the set-theoretic issue mentioned in Remark \ref{rmk:iso closed}.}
\end{enui}
\end{rmks}

Let $\G$ be a subset of $\caps(\CCC)$. We define the set CHLS generated by $\G$ as on p.~\pageref{CHLS gen set}, except that we ask the domain of each function $c$ to be $\OO_0$ rather than $\OO$. The reason for this change is the set-theoretic issue discussed in Remark \ref{rmk:iso closed}. (See also Remarks \ref{rmk:cap}.) As mentioned in Section \ref{sec:special cases}, in \cite[Problem 5, p.~17]{CHLS} K.~Cieliebak, H.~Hofer, J.~Latschev, and F.~Schlenk posed the following problem:
\begin{Pb}For a given symplectic category find a minimal CHLS capacity-generating set.
\end{Pb}

In particular, we may ask whether there exists a \emph{countable} generating set. Our first main result, Theorem \ref{thm:card cap gen} below, answers this question in the negative for a notion of a generation that, morally speaking, is much weaker than CHLS generation. The theorem states that in dimension at least 4 the cardinality of the set of generalized (or normalized) capacities on $\CCC$ is $\beth_2$, provided that the category contains a certain one-parameter family of objects.\footnote{It is formulated for a form category, not just for a symplectic category.} 
Its last part implies that a set of $[0,\infty]$-valued functions of cardinality at most $\beth_1$ countably Borel-generates a set of cardinality at most $\beth_1$ in the sense of Definition \ref{defi:gen set} below.

As an immediate consequence, every countably Borel-generating set for $\caps(\CCC)$ (or $\nc(\CCC)$) has cardinality bigger than the continuum. See Corollary \ref{cor:gen} below. Countable Borel-generation is a weak notion of generation. (Compare to Remark \ref{rmk:Borel} below.) It is weaker than the notion of limit-min/max generation (as defined on p.~\pageref{lim min max cap gen}).\footnote{This follows from the fact that a given set of capacities countably Borel-generates a larger set than it limit-min/max generates. Compare to Remark \ref{rmk:gen}.} 
Hence Corollary \ref{cor:gen} makes a statement about a large class of ``generating sets of capacities''.

This corollary diminishes the hope of finding manageable generating sets of (generalized) symplectic capacities.\\

To state our first main result, we need the following. Let $(X,\tau)$ be a topological space. Recall that the ($\tau$-)Borel $\si$-algebra of $X$ is the smallest $\si$-algebra containing the topology of $X$. We call its elements \emph{($\tau$-)Borel sets}. 
\begin{rmk}[Borel sets]\label{rmk:Borel}
\textnormal{Consider the real line $X=\R$. The axiom of choice (AC) implies that there exist subsets of $\R$ that are not Lebesgue-measurable, hence not Borel-measurable. However, all subsets occurring in practice are Borel. Furthermore, for any concretely described subset of $\R$, it appears to be difficult to prove (using AC) that it is indeed not Borel-measurable.\footnote{An example of such a subset $A$ was provided by N.~Luzin. It can be obtained from \cite[Exercise (27.2), p.~209]{Ke} via \cite[Exercise (3.4)(ii), p.~14]{Ke}. This set is $\boldsymbol{\Si}^1_1$-analytic, see \cite[Definitions (22.9), p.~169, (21.13), p.~156]{Ke}. It follows from a theorem of Souslin, \cite[(14.2) Theorem, p.~85]{Ke} and the definition of $\boldsymbol{\Si}^1_1$-analyticity that $A$ is not Borel.}
}
\end{rmk}
Let now $(X,\tau)$ and $(X',\tau')$ be topological spaces. \label{Borel}A map $f:X\to X'$ is called \emph{$(\tau,\tau')$-Borel-measurable} iff the pre-image under $f$ of every $\tau'$-Borel set in $X'$ is a $\tau$-Borel set.\footnote{This happens if and only if the pre-image under $f$ of every element of $\tau'$ is a $\tau$-Borel set.} In particular, every continuous map is Borel-measurable. Borel-measurability is preserved under composition. It is preserved under pointwise limits of sequences if $X'$ is metrizable. This yields many examples of Borel-measurable maps. In fact, all maps occurring in practice are Borel-measurable.

Let $S,S'$ be sets. We denote
\[{S'}^S:=\big\{\textrm{map from $S$ to }S'\big\}.\]
For every subset $\G\sub{S'}^S$ we denote by
\begin{equation}\label{eq:ev u}\ev_\G:S\to {S'}^\G,\quad\ev_\G(s)(u):=u(s),\end{equation}
the evaluation map. If $(X,\tau)$ is a topological space then we denote by $\tau_S$ the product topology on $X^S$. 
\begin{defi}[countably Borel-generated set]\label{defi:gen set}
Let $S$ be a set, $(X,\tau)$ a topological space, and $\G\sub X^S$. We define the \emph{set countably ($\tau$-)Borel-generated by $\G$} to be
\[\lan\G\ran:=\big\{f\circ\ev_{\G_0}\,\big|\,\G_0\sub\G\textrm{ countable, }f:X^{\G_0}\to X\textrm{: $(\tau_{\G_0},\tau)$-Borel-measurable}\big\}\sub X^S.\]
For every subset $\F\sub X^S$ we say that \emph{$\G$ countably ($\tau$-)Borel-generates at least $\F$} iff $\F\sub\lan\G\ran$.
\end{defi}

We denote by $\INT S$ the interior of a subset $S$ of a topological space. Let $V$ be a vector space, $S\sub V$, $A\sub\R$, and $n\in\N_0$. We denote $AS:=\big\{av\,\big|\,a\in A,\,v\in S\big\}$. In the case $A=\{a\}$ we also denote this set by $aS$. We call $S$ \emph{strictly starshaped around 0} iff $[0,1)S\sub\INT S$. For every $i\in\{1,\ldots,n\}$ we denote by $\pr_i:V^n=V\x\cdots\x V\to V$ the canonical projection onto the $i$-th component. For every multilinear form $\om$ on $V$ we denote
\[\om^{\oplus n}:=\sum_{i=1}^n\pr_i^*\om.\]
For every $r\in(1,\infty)$ we define the closed spherical shell of radii $1,r$ in $\R^m$ to be
\[\sh_r^m:=\BAR{B}^m_r\setminus B^m_1.\]
We equip $\sh_r:=\sh_r^{2n}$ with the standard symplectic form $\omst$. The first main result of this article is the following.
\begin{thm}[cardinalities of the set of (normalized) capacities and of the generated set]\label{thm:card cap gen} The following statements hold:
\begin{enui}
\item\label{thm:card cap gen:caps} Let $k,n\in\{2,3,\ldots\}$ with $k$ even, and $\CCC=(\OO,\M)$ be a $(kn,k)$-form category. Then the cardinalitiy of $\caps(\CCC)$ equals $\beth_2$, provided that there exist
\begin{itemize}
\item a (real) vector space $V$ of dimension $k$,
\item a volume form $\Om$ on $V$,\footnote{By this we mean a nonvanishing top degree skewsymmetric multilinear form.}
\item a nonempty compact submanifold $K$ of $V^n$ (with boundary) that is strictly starshaped around 0,
\item a number $r\in\left(1,\sqrt[kn]2\right)$,
\end{itemize}
such that defining $M_a:=(r+a)K\wo\INT K$ and equipping this manifold with the restriction of $\Om^{\oplus n}$, we have
\begin{equation}\label{eq:Ma disj M-a}M_a\disj M_{-a}\in\OO,\quad\forall a\in(0,r-1).\,\footnote{Here $A\disj B$ denotes the disjoint union of two sets $A,B$. This can be defined in different ways, e.g.~as the set consisting of all pairs $(0,a),(1,b)$, with $a\in A$, $b\in B$, or alternatively pairs $(1,a),(2,b)$. Based on this, we obtain two definitions of the disjoint union of two objects of $\OMEGA{kn}{k}$. The disjoint union defined in either way is isomorphic to the one defined in the other way. Since we assume $\CCC$ to be isomorphism-closed, condition \eqref{eq:Ma disj M-a} does not depend on the choice of how we define the disjoint union.}
\end{equation}
\item\label{thm:card cap gen:ncaps} Let $n\in\{2,3,\ldots\}$ and $\CCC=(\OO,\M)$ be a $(2n,2)$-form category that contains the objects $B$ and $Z$. The cardinality of $\nc(\CCC)$ equals $\beth_2$, provided that there exists $r\in\left(1,\sqrt[2n]2\right)$ satisfying
\begin{equation}\label{eq:sh}\sh_{r-a}\disj\sh_{r+a}\in\OO,\quad\quad\forall a\in(0,r-1).\end{equation}
\item\label{thm:card cap gen:gen} Let $S$ be a set and $(X,\tau)$ a separable metrizable topological space. If a subset of $X^S$ has cardinality at most $\beth_1$, then the set it countably $\tau$-Borel-generates has cardinality at most $\beth_1$.
\end{enui}  
\end{thm}
This result has the following immediate application. We define $\OO_0$ as in \reff{eq:OO0}, and $\tau_0$ to be the standard topology on $[0,\infty]$, \wrt which it is homeomorphic to the interval $[0,1]$.
\begin{cor}[cardinality of a generating set]\label{cor:gen}\begin{enui}
\item\label{cor:gen:caps} Under the hypotheses of Theorem \ref{thm:card cap gen}\reff{thm:card cap gen:caps} every subset of $[0,\infty]^{\OO_0}$ that countably $\tau_0$-Borel-generates at least $\caps(\CCC)$ has cardinality bigger than $\beth_1$.
\item\label{cor:gen:ncaps} Under the hypotheses of Theorem \ref{thm:card cap gen}\reff{thm:card cap gen:ncaps} every subset of $[0,\infty]^{\OO_0}$ that countably $\tau_0$-Borel-generates at least $\nc(\CCC)$ has cardinality bigger than $\beth_1$.
\end{enui}
\end{cor}
This corollary answers Question \ref{q:count gen} (p.~\pageref{q:count gen}) negatively for every symplectic category satisfying the hypotheses of Theorem \ref{thm:card cap gen}\reff{thm:card cap gen:caps}. An example of such a category is the category of all symplectic manifolds of some fixed dimension, which is at least 4.

\newpage
\begin{Rmks}[cardinalities of the set of (normalized) capacities and of the generated set]\textnormal{\begin{itemize}
\item As Corollary \ref{cor:gen} holds for $(kn,k)$-form categories with $k$ even and $n\geq2$, the fact that generating sets of capacities are large, is not a genuinely symplectic phenomenon.
\item The proof of Theorem \ref{thm:card cap gen}\reff{thm:card cap gen:ncaps} shows that the cardinality of the set of \emph{discontinuous} normalized capacities is $\beth_2$. This improves the result of K.~Zehmisch and the second author that discontinuous capacities exist, see \cite{ZZ}.%
\footnote{The proof of \cite[Theorem 1.2]{ZZ} actually shows that the spherical shell capacities used in that proof are all different. This implies that the set of discontinuous normalized symplectic capacities has cardinality at least $\beth_1$.}
\item The statements of Theorem \ref{thm:card cap gen}(\ref{thm:card cap gen:caps},\ref{thm:card cap gen:ncaps}) and thus of Corollary \ref{cor:gen} hold in a more general setting, see Theorem \ref{thm:card cap hel} and Proposition \ref{prop:I collection} below. In particular, let $V,\Om$ be as in Theorem \ref{thm:card cap gen}\reff{thm:card cap gen:caps}, $j\in\{1,2,\ldots\}$, and for each $a\in\R$ let $M_a$ be the complement of $j$ disjoint connected open sets in some compact submanifold of $V^n$. The cardinality of $\caps(\CCC)$ equals $\beth_2$, provided that $M_a\disj M_{-a}\in\OO$ %
\footnote{In particular we assume here that $M_a$ is a smooth submanifold of $V^n$.} %
for every $a$, the volumes of the open sets are all equal (also for different $a$), the volume of each $M_a$ is small enough and strictly increasing in $a$, the infimum of these volumes is positive, and each $M_a$ is 1-connected.
\item Morally, Corollary \ref{cor:gen} implies that every generating set of capacities has as many elements as there are capacities. More precisely, we denote by ZF the Zermelo-Fraenkel axiomatic system, and ZFC := ZF + AC. We claim that ZFC is consistent with the statement that under the hypotheses of Theorem \ref{thm:card cap gen}\reff{thm:card cap gen:caps} every subset of $[0,\infty]^{\OO_0}$ that countably Borel-generates at least $\caps(\CCC)$ has the same cardinality as $\caps(\CCC)$ (namely $\beth_2$) \footnote{provided that ZF is consistent}.\\ 
To see this, assume that the generalized continuum hypothesis (GCH) holds. This means that for every infinite cardinal $\al$ there is no cardinal strictly between $\al$ and $2^\al$. In particular, there is no cardinal strictly between $\beth_1$ and $\beth_2=2^{\beth_1}$. Hence under the hypotheses of Theorem \ref{thm:card cap gen}\reff{thm:card cap gen:caps} by Corollary \ref{cor:gen}\reff{cor:gen:caps} every subset of $[0,\infty]^{\OO_0}$ that countably Borel-generates at least $\nc(\CCC)$ has cardinality at least $\beth_2$. By Theorem \ref{thm:card cap gen}\reff{thm:card cap gen:caps} this equals the cardinality of $\caps(\CCC)$. Since GCH is consistent with ZFC \footnote{provided that ZF is consistent}, the claim follows.
\end{itemize}
}
\end{Rmks}
\begin{rmk}[comparison of different notions of generating sets]\label{rmk:gen}
\textnormal{Let $\CCC$ be a symplectic category and $\G$ a CHLS generating set (as defined on p.~\pageref{CHLS cap gen}), with the extra condition that each combining function $f:[0,\infty]^\ell\to[0,\infty]$ is Borel-measurable. Then $\G$ countably Borel-generates $\caps(\CCC)$. (See Definition \ref{defi:gen set}.)\footnote{To see this, let $c\in\caps(\CCC)$. We choose a sequence of Borel-measurable combining functions and finite sets of generalized capacities in $\G$ as in the definition of a CHLS generating set. We define $\G_0$ to be the set of all capacities occurring in the sequence. Each combining function gives rise to a Borel-measurable function from $[0,\infty]^{\G_0}$ to $[0,\infty]$. Its restriction to the image of $\ev_{\G_0}$ is measurable \wrt the $\si$-algebra induced by the Borel $\si$-algebra. By assumption the sequence of these restrictions converges pointwise. The limit $f$ is again measurable. Since its target space is $[0,\infty]$, an argument involving approximations by simple functions shows that $f$ extends to a Borel-measurable function on $[0,\infty]^{\G_0}$. Hence $\G_0$ and $f$ satisfy the conditions of Definition \ref{defi:gen set}, as desired.} 
This holds in particular if $\G$ is limit-min/max generating (as defined on p.~\pageref{lim min max cap gen}).
}

\textnormal{Definition \ref{defi:gen set} relaxes the conditions in the definition of a CHLS generating set in two ways:}
\begin{itemize}
\item\textnormal{The combining functions are allowed to depend on countably many variables (elements of the generating set), not just on finitely many variables.}
\item\textnormal{The assumption that the combining functions are homogeneous and monotone is omitted.}
\end{itemize}
\end{rmk}
\subsection{Representability of symplectic capacities and morphism detection}
The following corollary is a direct consequence of Theorem \ref{thm:card cap gen}\reff{thm:card cap gen:ncaps}. Let $\CCC$ be an $(m,k)$-form category. We say that a generalized capacity $c$ on $\CCC$ is \emph{domain-/ target-representable} iff there exists an object $(M,\om)$ of $\OMEGA{m}{k}$, such that $c=c_{(M,\om)}$/ $c=c^{(M,\om)}$ (defined as in Example \ref{xpl:emb cap}). We say that \emph{almost no} element of a given infinite set has a given property iff the subset of all elements with this property has smaller cardinality than the whole set.
\begin{cor}[representability]\label{cor:rep} Under the hypotheses of Theorem \ref{thm:card cap gen}\reff{thm:card cap gen:ncaps} almost no normalized capacity on $\CCC$ is domain- or target-representable.
\end{cor}
It follows that under the hypotheses of Theorem \ref{thm:card cap gen}\reff{thm:card cap gen:ncaps} there are as many normalized capacities that are neither domain- nor target-representable, as there are normalized capacities overall (namely $\beth_2$). This provides some answers to Questions \ref{q:rep to},\ref{q:rep from}, and \ref{q:rep every} (p.~\pageref{q:rep to}). The statement of Corollary \ref{cor:rep} holds in particular for $\CCC$ given by the category of all symplectic manifolds of some fixed dimension, which is at least 4.
\begin{proof}[Proof of Corollary \ref{cor:rep}] The set of isomorphism classes of $2n$-dimensional manifolds together with 2-forms has cardinality $\beth_1$. This follows from Corollary \ref{cor:card sec} on p.~\pageref{cor:card sec}. The image of this set under the map $[(M,\om)]\mapsto c_{M,\om}$ is the set of all domain-representable capacities. It follows that at most $\beth_1$ normalized capacities are domain-representable. A similar statement holds for target-representation. The statement of Corollary \ref{cor:rep} now follows from Theorem \ref{thm:card cap gen}\reff{thm:card cap gen:ncaps}.
\end{proof}
The proof technique for Corollary \ref{cor:gen} can potentially also be used to show that certain sets of capacities do not detect morphisms. To explain this, let $\CCC=(\OO,\M)$ be an $(m,k)$-form category and $\G\sub\caps(\CCC)$. We say that $\G$ detects morphisms iff for each pair of objects $(M,\om),(M',\om')$ of $\CCC$ the following holds. Assume that $c(M,\om)\leq c(M',\om')$, for every $c\in\G$. Then there exists a $\CCC$-morphism from $(M,\om)$ to $(M',\om')$. CHLS asked the following question in the case in which $\CCC$ is a symplectic category and $\G=\caps(\CCC)$ (see \cite[Question 1, p.~20]{CHLS}):
\begin{q}\label{q:det} Does $\G$ detect morphisms?\footnote{CHLS do not use our terminology of ``morphism detection''. Instead, the title of the subsection in which they ask their Question 1, is ``Recognition''. We think that the expression ``$\G$ detects morphisms'' more accurately describes the condition that $\G$ determines whether there exists a morphism between two given objects.}
\end{q}
\begin{Rmk}[monotone generation and detection of morphisms]
\textnormal{We equip the set $S:=\OO_0$ with the pre-order given by $(M,\om)\leq(M',\om')$ iff there exists a $\CCC$-morphism from $(M,\om)$ to $(M',\om')$. We also equip this set with the $(0,\infty)$-action given by rescaling of forms. Suppose the following:
\begin{enumerate}
\item[(*)] Every subset of $\caps(\CCC)$ that monotonely generates in the sense of Definition \ref{defi:mon gen} on p.~\pageref{defi:mon gen}, has cardinality bigger than $\beth_1$. 
\end{enumerate}
Let $\G\sub\caps(\CCC)$ be a subset of cardinality at most $\beth_1$. Then $\G$ does not detect morphisms, thus the answer to Question \ref{q:det} is ``no''. To see this, observe that by our assumption (*) the set $\G$ does not monotonely generate. Therefore, by Proposition \ref{prop:mon gen} on p.~\pageref{prop:mon gen}, it is not almost order-reflecting. Hence $\G$ is not order-reflecting, i.e., it does not detect morphisms.
}

\textnormal{By Corollary \ref{cor:gen}, under the hypotheses of Theorem \ref{thm:card cap gen}\reff{thm:card cap gen:caps}, condition (*) is satisfied with ``monotonely generates'' replaced by ``countably $\tau_0$-Borel-generates''. Therefore potentially, the proof technique for Corollary \ref{cor:gen} may be adapted, in order to provide a negative answer to Question \ref{q:det} under suitable conditions on $\CCC$ that do not involve (*), if the cardinality of $\G$ is at most $\beth_1$.
}
\end{Rmk}
\subsection{Uncountability of every generating set under a mild hypothesis}
Our second main result states that every finitely differentiably capacity-generating set for a small weak form category is uncountable, if the category contains a one-parameter family of objects that is ``strictly volume-increasing'' and ``embedding-capacity-wise constant''. 
Morally speaking, this hypothesis is weaker than those of the first main result. In order to state the result, we recall the notion of finite differentiable generation from Section \ref{sec:special cases}, reformulating and generalizing it slightly:
\begin{Defi}[finite differentiable generation] Let $S$ be a set, and $\F,\G\subseteq[0,\infty]^S$. We say that $\G$ \emph{finitely differentiably generates at least $\F$} iff the following holds. For every $F\in\F$ there exists a finite subset $\G_0\subseteq\G$ and a differentiable function $f:[0,\infty]^{\G_0}\to[0,\infty]$ \footnote{Here we view $[0,\infty]$ as a compact 1-dimensional manifold with boundary. The set $[0,\infty]^{\G_0}$ carries a canonical structure of a smooth finite-dimensional manifold with boundary and corners. The function $f$ is only assumed to be differentiable one time, with possibly discontinuous derivative.} 
, such that $F=f\circ ev_{\G_0}$.
\end{Defi}
Let now $k,n\in\N:=\{1,2,\ldots\}$ and $(M,\om)$ be an object of $\OMEGA{kn}{k}$. We call $\om$ \emph{maxipotent} iff $\om^{\wedge n}=\om\wedge\cdots\wedge\om$ does not vanish anywhere.
\begin{rmk}[maxipotency and nondegeneracy]\label{rmk:nondeg} \textnormal{Let $V$ be a (real) vector space and $k\in\N$. We call a $k$-linear form $\om$ on $V$ \emph{nondegenerate} iff interior multiplication with $\om$ is an injective map from $V$ to the space of $(k-1)$-linear forms. Let $k,n\in\N$ and assume that $\dim V=kn$. We call a skewsymmetric $k$-form $\om$ on $V$ \emph{maxipotent} iff $\om^{\wedge n}\neq0$. Every maxipotent form on $V$ is nondegenerate. The converse holds if and only if $k=1$, $k=2$, or $n=1$.
}
\end{rmk}
Let $(M,\om)$ be a maxipotent object of $\OMEGA{kn}{k}$. We equip $M$ with the orientation induced by $\om^{\wedge n}$ and define
\begin{equation}\label{eq:Vol M om}\Vol(M):=\Vol(M,\om):=\frac1{n!}\int_M\om^{\wedge n}.\end{equation}
\begin{rmk}[volume]\label{rmk:vol}\textnormal{Assume that $k$ is odd. Then we have $\om\wedge\om=0$, and therefore $\Vol(M,\om)=0$ in the case $n\geq2$.}
\end{rmk}
Let $\CCC=(\OO,\M)$ be a small weak $(kn,k)$-form category. We define the notion of a generalized capacity as in Definition \ref{defi:cap}.
\begin{rmk}\label{rmk:defi cap}\textnormal{Since $\CCC$ is small, Definition \ref{defi:cap} and our original Definition \ref{defi:cap } are equivalent in the sense that capacities in either sense correspond to each other in a canonical way. This follows from Remark \ref{rmk:cap}\reff{rmk:cap:OO0}. Recall here that a capacity in the sense of Definition \ref{defi:cap} (Definition \ref{defi:cap }) is a function with domain $\OO_0$ ($\OO$).
}
\end{rmk}
Our second main result is the following.
\begin{thm}[uncountability of every generating set under a mild hypothesis]\label{thm:gen uncount} Every subset of $\caps(\CCC)$ that finitely differentiably generates (at least) $\caps(\CCC)$, is uncountable, provided that there exists an interval $A$ of positive length and a map $M:A\to\OO$, such that
\begin{eqnarray}\label{eq:nondeg}&M_a:=M(a)\textrm{ is maxipotent for every }a\in A,&\\
\label{eq:Vol M}&\Vol\circ M\textrm{ is continuous and strictly increasing,}&\\
\label{eq:c Ma Ma'}&c_{M_a}(M_{a'})=1,\forall a,a'\in A:\,a\leq a'.&
\end{eqnarray}
\end{thm}
\begin{Rmks}\textnormal{\begin{itemize}\item Condition \reff{eq:nondeg} ensures that the volume of each $M_a$ is well-defined. Hence condition \reff{eq:Vol M} makes sense.
\item Condition \reff{eq:c Ma Ma'} means that $M$ is ``embedding-capacity-wise constant'', in the sense that the composition of the map $\big\{(a,a')\in A^2\,\big|\,a\leq a'\big\}\ni(a,a')\mapsto(M_a,M_{a'})$ with the ``embedding capacity function'' $(X,X')\mapsto c_X(X')$ is constant.
\item Assume that there exists a map $M$ satisfying (\ref{eq:nondeg},\ref{eq:Vol M}). Then we have $n>0$. If $n\geq2$, then $k$ is even. This follows from Remark \ref{rmk:vol}. Assume that there exists a map satisfying (\ref{eq:nondeg},\ref{eq:Vol M},\ref{eq:c Ma Ma'}). Then we have $k>0$. If each $M_a$ is compact, then $n\neq1$. This follows from Moser's isotopy argument.
\end{itemize}}
\end{Rmks}
\begin{xpl}\label{xpl:increasing}\textnormal{Let $n\geq2$ and $A$ be an interval of positive length. We denote by $\U$ the set of all open subsets of $\R^{2n}$ that contain $B^{2n}_1$ and are contained in $Z^{2n}_1$. We equip each element of $\U$ with the restriction of the form $\omst$. Let $M:A\to\U$ be an increasing map in the sense that $a\leq a'$ implies that $M(a)\sub M(a')$. If $M$ also satisfies \reff{eq:Vol M} then it satisfies all conditions of Theorem \ref{thm:gen uncount}. The inequality ``$\leq$'' in condition \reff{eq:c Ma Ma'} follows from Gromov's nonsqueezing theorem.
}
\end{xpl}
As an application of Theorem \ref{thm:gen uncount} we obtain Theorem \ref{thm:ell} (p.~\pageref{thm:ell}, uncountability of every generating set for ellipsoids).
\begin{Rmk}\textnormal{Theorem \ref{thm:ell} is concerned with the \emph{weak} symplectic category $\ELL{V}$ of ellipsoids in $V$. This is not a symplectic category, since it is not isomorphism-closed in $\OMEGA{2n}{2}$, where $2n:=\dim V$. This is the reason for formulating Theorem \ref{thm:gen uncount} for a (small) \emph{weak} form category.
}
\end{Rmk}
\begin{Rmks}\begin{itemize}\item\textnormal{The hypotheses of Corollary \ref{cor:gen}\reff{cor:gen:caps} (i.e., of Theorem \ref{thm:card cap gen}\reff{thm:card cap gen:caps}) and of Theorem \ref{thm:gen uncount} do not imply each other. However, morally, the hypotheses of Corollary \ref{cor:gen}\reff{cor:gen:caps} are more restrictive than that of Theorem \ref{thm:gen uncount}. This becomes literally true if we modify the hypotheses of this corollary by replacing the disjoint union $M_a\disj M_{-a}$ by $M_a$.}
\item\textnormal{On the other hand, the conclusion of Corollary \ref{cor:gen}\reff{cor:gen:caps} is stronger than that of Theorem \ref{thm:gen uncount}.}
\end{itemize}
\end{Rmks}
\subsection{Ideas of proof}\label{subsec:ideas}
The idea of the proof of Theorem \ref{thm:card cap gen}\reff{thm:card cap gen:caps} is the following. Recall the definition \reff{eq:c M om} of the domain-embedding capacity $c_M:=c_{M,\om}$. We choose $V,\Om,K,r$ and define $M_a$ as in the hypothesis of the theorem. We define $W_a:=M_a\disj M_{-a}$. For each $A\in\PP((0,r-1))$ \footnote{Here $\PP(S)$ denotes the power set of a set $S$.} we define
\[c_A:=\sup_{a\in A}c_{W_a}.\]
This is a capacity, satisfying
\begin{eqnarray}
\label{eq:cA}&c_A(W_a)=1,\quad\forall a\in A,&\\
\label{eq:sup cA}&\sup_{a\in(0,r-1)\wo A}c_A(W_a)<1.&
\end{eqnarray}
The second statement follows from Stokes' Theorem for helicity. Helicity assigns a real number to an exact $k$-form on a closed oriented manifold of dimension $kn-1$, where $n\geq2$. (To build some intuition, see the explanations on p.~\pageref{helicity argument} and Figures \ref{fig:big wrap},\ref{fig:not wrap},\ref{fig:small wrap} on p.~\pageref{fig:big wrap}--\pageref{fig:small wrap}.) Helicity generalizes contact volume. The conditions (\ref{eq:cA},\ref{eq:sup cA}) imply that $c_A\neq c_{A'}$ if $A\neq A'\in\PP((0,r-1))$. Since the cardinality of $\PP((0,r-1))$ equals $\beth_2$, it follows that the cardinality of $\caps(\CCC)$ is at least $\beth_2$.

On the other hand, we denote by $S$ the set of isomorphism classes of objects of $\CCC$. This set has cardinality $\beth_1$. Since $\caps(\CCC)$ can be viewed as a subset of $[0,\infty]^S$, it has cardinality at most $\beth_2$, hence equal to $\beth_2$.

A refined version of this argument shows Theorem \ref{thm:card cap gen}\reff{thm:card cap gen:ncaps}, i.e., that $\vert\nc(\CCC)\vert=\beth_2$. For this we normalize each capacity $c_A$, by replacing it by the maximum of $c_A$ and the Gromov width.
\begin{Rmk}[helicity argument] \textnormal{In \cite{ZZ} K.~Zehmisch and F.~Ziltener used helicity to show that the spherical capacity is discontinuous on some smooth family of ellipsoidal shells. This argument is related to the proof of Theorem \ref{thm:card cap gen}(\ref{thm:card cap gen:caps},\ref{thm:card cap gen:ncaps}).}
\end{Rmk}
The proof of Theorem \ref{thm:card cap gen}\reff{thm:card cap gen:gen} is based on the fact that the set of Borel-measurable maps from a second countable space to a separable metrizable space has cardinality at most $\beth_1$. The proof of this uses the following well-known results:
\begin{itemize}
\item Every map $f$ with target a separable metric space is determined by the pre-images under $f$ of balls with rational radii around points in a countable dense subset.
\item The $\si$-algebra generated by a collection of cardinality at most $\beth_1$ has itself cardinality at most $\beth_1$. The proof of this uses transfinite induction.\\
\end{itemize}

The idea of the proof of Theorem \ref{thm:gen uncount} is to use Lebesgue's Monotone Differentiation Theorem, which states that every monotone function on an interval is differentiable almost everywhere. It follows that for every countable set $\G$ of capacities, there exists a point $a_0\in A$ at which the function $c\circ M$ is differentiable, for every $c\in\G$. On the other hand, our conditions on the map $M:A\to\OO$ imply that the function $c_{M_{a_0}}\circ M$ is not differentiable at $a_0$. It follows that $\G$ does not finitely differentiably generate $c_{M_{a_0}}$.
\begin{Rmk}[diagonal argument]\textnormal{This idea of the proof is remotely reminiscent of Cantor's second diagonal argument, which shows that the set of real numbers is uncountable. Namely, consider the open sentence $P$ given by:}

\noindent\textnormal{$P(a,b)$: ``The function $c_{M_a}\circ M$ is differentiable at $b$.''}

\textnormal{The proof of Theorem \ref{thm:gen uncount} exploits the fact that $P$ is false along the diagonal, that is, $P(a,a)$ is false for all $a$.}
\end{Rmk}
\subsection{McDuff's characterization of existence of symplectic embeddings for ellipsoids in dimension 4 and monotone generation}
In this subsection we recall a result of D.~McDuff, which states that the ECH-capacities are monotonely generating for the category of ellipsoids in dimension 4. On ellipsoids, these capacities are given by the following. Let $n,j\in\N_0$. We define the function
\begin{eqnarray*}&\NN^n_j:[0,\infty)^n\to[0,\infty),&\\
&\NN^n_j(a):=\min\left\{b\in[0,\infty)\,\Bigg|\,j+1\leq\#\left\{m\in\N_0^n\,\Bigg|\,m\cdot a=\displaystyle\sum_{i=1}^nm_ia_i\leq b\right\}\right\}.&
\end{eqnarray*}
\begin{Rmk}\textnormal{The sequence $\big(\NN^n_j(a)\big)_{j\in\N_0}$ is obtained by arranging all the nonnegative integer combinations of $a_1,\ldots,a_n$ in increasing order, with repetitions.}
\end{Rmk}
We define the ellipsoid
\[E(a):=\left\{x=\big(x_1,\ldots,x_n)\in\R^{2n}=(\R^2)^n\,\Bigg|\,\sum_{i=1}^n\frac{\Vert x_i\Vert^2}{a_i}<1\right\}.\]
(Here $\Vert\cdot\Vert$ denotes the Euclidean norm on $\R^2$.) We equip this manifold with the standard symplectic form. 

Let $V:=(V,\om)$ be a symplectic vector space. We denote by $\OEll{V}$ the set of all pairs $(E,\om|E)$, where $E$ is a (bounded, open, full) ellipsoid in $V$, and by $\MEll{V}$ the set of all symplectic embeddings between elements of $\OEll{V}$. We define the \emph{category of (open) ellipsoids in $V$} to be the pair $\Ell{V}:=(\OEll{V},\MEll{V})$. This is a \emph{small} weak symplectic category. For such a category we may view a generalized capacity as a monotone and conformal function on the \emph{whole} set of objects. (Compare to Remark \ref{rmk:defi cap}.) For every $j\in\N_0$ we define the function
\begin{equation}\label{eq:ck}c^V_j:\OEll{V}\to[0,\infty),\end{equation}
by setting $c^V_j(E):=\NN^n_j(a)$, where $a\in[0,\infty)^n$ is such that $E$ is affinely symplectomorphic to $E(a)$. This number is well-defined, i.e., such an $a$ exists (see \cite[Lemma 2.43]{MS}) and $\NN^n_j(a)$ does not depend on its choice. The latter is true, since if $E(a)$ and $E(a')$ are affinely symplectomorphic, then $a$ and $a'$ are permutations of each other. (See \cite[Lemma 2.43]{MS}.) The following result is due to M.~Hutchings.
\begin{thm}[ECH-capacities]\label{thm:ck} If $\dim V=4$ then for every $j\in\N_0$ the function $c^V_j$ is a generalized capacity on $\Ell{V}$.
\end{thm}
\begin{proof} Conformality follows from the definition of $\NN^n_j$. Monotonicity was proved by M.~Hutchings in \cite[Proposition 1.2, Theorem 1.1]{Hu}.
\end{proof}
\begin{Rmks}\textnormal{\begin{itemize}\item The function $c^V_j$ is the restriction of the $j$-th ECH-capacity to $\Ell{V}$, see \cite[Proposition 1.2]{Hu}.
\item The category $\ELL{V}$ that we considered on p.~\pageref{ELL V} is a subcategory of $\Ell{V}$.
\end{itemize}}
\end{Rmks}
McDuff proved that the set of all $c^V_j$ (with $j\in\N_0$) monotonely generates all generalized capacities. To explain this, we equip the interval $(0,\infty)$ with multiplication and let it act on the extended interval $[0,\infty]$ via multiplication. Let $S,S'$ be sets. We fix $(0,\infty)$-actions on $S$ and $S'$ and call a map $f:S\to S'$ \emph{(positively 1-)homogeneous} iff it is $(0,\infty)$-equivariant. 

Recall that a preorder on a set $S$ is a reflexive and transitive relation on $S$. We call a map $f$ between two preordered sets \emph{monotone (or increasing)} if it preserves the preorders, i.e., if $s\leq s'$ implies that $f(s)\leq f(s')$. Let $(S,\leq)$ be a preordered set. We fix an order-preserving $(0,\infty)$-action on $S$. We define the set of \emph{(generalized) capacities on $S$} to be
\begin{equation}\label{eq:caps S}\caps(S):=\big\{c\in[0,\infty]^S\,\big|\,c\textrm{ monotone and $(0,\infty)$-equivariant}\big\}.\end{equation}
We equip the set $[0,\infty]^S$ with the preorder
\[x\leq x'\iff x(s)\leq x'(s),\,\forall s\in S.\]
Let $\G\sub\caps(S)$.
\begin{defi}[monotone and homogeneous monotone generation]\label{defi:mon gen} We say that $\G$ \emph{monotonely (capacity-)generates} iff for every $c\in\caps(S)$ there exists a monotone function $F:[0,\infty]^\G\to[0,\infty]$, such that $c=F\circ\ev_\G$. We say that $\G$ \emph{homogeneously and monotonely (capacity-)generates} iff the function $F$ above can also be chosen to be homogeneous.
\end{defi}
\begin{Rmk}[monotone versus homogeneous and monotone generation]\textnormal{The set $\G$ monotonely generates if and only if it homogeneously and monotonely generates. The ``only if''-direction follows by considering the monotonization (see p.~\pageref{mon} below) of the restriction of a function $F$ as in Definition \ref{defi:mon gen} to the image of $\ev_\G$. Here we use that every $c\in\caps(S)$ is homogeneous, and thus $F|\im(\ev_\G)$ is homogeneous, as well as Remark \ref{rmk:mon} on p.~\pageref{rmk:mon}.}
\end{Rmk}
Let $V:=(V,\om)$ be a symplectic vector space. Recall the definition \eqref{eq:ck} of the capacity $c^V_j$. The next result easily follows from McDuff's solution of the embedding problem for ellipsoids in dimension 4. (See Section \ref{sec:gen ell}.)
\begin{thm}[monotone generation for ellipsoids in dimension 4]\label{thm:gen ell} If $\dim V=4$ then the set of all $c^V_j$ (with $j\in\N_0$) monotonely generates (the generalized capacities on the category of ellipsoids $\Ell{V}$).
\end{thm}
This theorem provides a positive answer to the variant of Question \ref{q:count gen} with ``limit-min/max generating'' replaced by ``monotonely generating''. Monotone generation is (possibly nonstrictly) weaker than CHLS generation, since the pointwise limit of monotone functions is monotone. To deduce the theorem from McDuff's result, we will characterize monotone generation in terms of almost order-reflexion.
\section{Proof of Theorem \ref{thm:card cap gen}(\ref{thm:card cap gen:caps},\ref{thm:card cap gen:ncaps}) 
(cardinality of the set of capacities)} \label{proof:thm:card cap gen:caps}

In this section we prove Theorem \ref{thm:card cap gen}(\ref{thm:card cap gen:caps},\ref{thm:card cap gen:ncaps}), based on a more general result, Theorem \ref{thm:card cap hel} below. That result states that the set of generalized capacities on a given $(kn,k)$-category $\CCC$ has cardinality $\beth_2$, provided that $\CCC$ contains each disjoint union $M_a\disj M_{-a}$ for a suitable one-parameter family of manifolds with forms $(M_a,\om_a)_{a\in A_0}$. A crucial hypothesis is that the collection of boundary helicities associated with $(M_a,\om_a)_{a\in A_0}$, is an $I$-collection. 

We also state Proposition \ref{prop:I collection}, which provides sufficient conditions for this hypothesis to be satisfied. 
\subsection{(Boundary) helicity of an exact differential form}
In this subsection we introduce the notion of helicity of an exact form, and based on this, the notion of boundary helicity.

Let $k,n\in\N_0$ be such that $n\geq2$, $N$ a closed\footnote{This means compact and without boundary.} $(kn-1)$-manifold, $\oo$ an orientation on $N$, and $\si$ an exact $k$-form on $N$.
\begin{defi}[helicity]\label{defi:hel} We define the helicity of $(N,\oo,\si)$ to be the integral
\begin{equation}\label{eq:h N oo om}h(N,\oo,\si):=\int_{N,\oo}\al\wedge\si^{\wedge(n-1)},\end{equation}
where $\al$ is an arbitrary primitive of $\si$, and $\int_{N,\oo}$ denotes integration over $N$ \wrt $\oo$.
\end{defi}
We show that this number is well-defined, i.e., it does not depend on the choice of the primitive $\al$. Let $\al$ and $\al'$ be primitives of $\si$. Then $\al'-\al$ is closed, and therefore 
\[(\al'-\al)\wedge\si^{\wedge(n-1)}=(-1)^{k-1}d\left((\al'-\al)\wedge\al\wedge\si^{\wedge(n-2)}\right).\]
Here we used that $n\geq2$. Using Stokes' Theorem and our assumption that $N$ has no boundary, it follows that 
\[\int_{N,\oo}(\al'-\al)\wedge\si^{\wedge(n-1)}=0.\]
Therefore, the integral (\ref{eq:h N oo om}) does not depend on the choice of $\al$.

\begin{rmk}[case $k$ odd, case $n=1$]\textnormal{The helicity vanishes if $k$ is odd. This follows from the equality
\[\al\wedge(d\al)^{n-1}=\frac12d\left(\al^{\wedge2}\wedge(d\al)^{n-2}\right),\]
which holds for every even-degree form $\al$, and from Stokes' Theorem. The helicity is not well-defined in the case $n=1$. Namely, in this case $\dim N=k-1$, and therefore every $(k-1)$-form is a primitive of the $k$-form 0. Hence the integral \reff{eq:h N oo om} depends on the choice of a primitive.}
\end{rmk}

\begin{rmk}[orientation]\label{rmk:or} \textnormal{Denoting by $\BAR\oo$ the orientation opposite to $\oo$, we have
\[h(N,\BAR\oo,\si)=-h(N,\oo,\si).\]
}
\end{rmk}

\begin{rmk}[rescaling]\label{rmk:resc} \textnormal{For every $c\in\R$ we have
\[h\big(N,\oo,c\si\big)=c^nh\big(N,\oo,\si\big).\]
This follows from a straight-forward argument.}
\end{rmk}

\begin{rmk}[naturality]\label{rmk:hel} \textnormal{Let $N$ and $N'$ be closed $(kn-1)$-manifolds, $\oo$ an orientation on $N$, $\si$ an exact $k$-form on $N$, and $\phi:N\to N'$ a (smooth) embedding. We denote
\[\phi_*(N,\oo,\si):=\big(\phi(N),\phi_*\oo,\phi_*\si\big)\]
(push-forwards of the orientation and the form). A straight-forward argument shows that 
\[h\big(\phi_*\big(N,\oo,\si\big)\big)=h(N,\oo,\si).\]
}
\end{rmk}

\begin{rmk}[helicity of a vector field] \textnormal{In the case $k=2$ and $n=2$ the integral (\ref{eq:h N oo om}) equals the helicity of a vector field $V$ on a three-manifold $N$, which is dual to the two-form $\si$, via some fixed volume form. See \cite[Definition 1.14, p. 125]{Arnold-Khesin}. This justifies the name ``helicity'' for the function $h$.}
\end{rmk}
The helicity of the boundary of a compact manifold equals the volume of the manifold. This is a crucial ingredient of the proofs of the main results and the content of the following lemma. Let $M$ be a manifold, $N\sub M$ a submanifold, and $\om$ a differential form on $M$. We denote by $\bd M$ the boundary of $M$, and
\begin{equation}\label{eq:om N}\om_N:=\textrm{pullback of $\om$ by the canonical inclusion of $N$ into $M$.}\end{equation}
If $\oo$ is an orientation on $M$ and $N$ is contained in $\bd M$, then we define
\begin{equation}\label{eq:oo N}\oo_N:=\oo_N^M:=\textrm{orientation of $N$ induced by $\oo$.}\end{equation}
Let $k,n\in\N_0$, such that $n\geq2$, $(M,\oo)$ be a compact oriented (smooth) manifold of dimension $kn$ and $\om$ an exact $k$-form on $M$.
\begin{lemma}[Stokes' theorem for helicity]\label{le:vol hel}
The following equality holds:
\[\int_{M,\oo}\om^{\wedge n}=h\big(\bd M,\oo_{\bd M},\om_{\bd M}\big).\]
\end{lemma}
\begin{Rmk}\textnormal{The left hand side of this equality is $n!$ times the signed volume of $M$ associated with $\oo$ and $\om$.}
\end{Rmk}
\begin{proof}[Proof of Lemma \ref{le:vol hel}] Choosing a primitive $\al$ of $\om$, we have
\[\om^{\wedge n}=d\big(\al\wedge\om^{\wedge(n-1)}\big),\]
and therefore, by Stokes' Theorem,
\[\int_{M,\oo}\om^{\wedge n}=\int_{\bd M,\oo_{\bd M}}\al\wedge\om^{\wedge(n-1)}=h(\bd M,\oo_{\bd M},\om_{\bd M}).\]
This proves Lemma \ref{le:vol hel}.
\end{proof}
This lemma has the following consequence. We denote
\[I_M:=\big\{\textrm{connected component of }\bd M\big\}.\]
\begin{defi}[boundary helicity] \label{defi:boundary hel} We define the \emph{boundary helicity of $(M,\oo,\om)$} to be the function
\[h_M:=h_{M,\oo,\om}:I_M\to\R,\quad h_{M,\oo,\om}(i):=h\big(i,\oo_i,\om_i\big),\]
\end{defi}
\begin{cor}[Stokes' theorem for helicity]\label{cor:vol hel} The following equality holds:
\[\int_{M,\oo}\om^{\wedge n}=\sum_{i\in I_M}h\big(i,\oo_i,\om_i\big).\]
\end{cor}
\begin{proof} This directly follows from Lemma \ref{le:vol hel}.
\end{proof}
\subsection{$I$-collections}
An $I$-collection is collection $f=(f_a)_{a\in A_0}$ of real-valued functions with finite domains, such that the supremum of a certain set of numbers is less than 1. The set consists of all numbers $C$ for which $A\cup B$ is nonempty, where $A$ and $B$ are certain sets of partitions, which depend on $f$ and $C$. $I$-collections will occur in the generalized main result, Theorem \ref{thm:card cap hel} below. Namely, one hypothesis of this result is that the boundary helicities of a certain collection of manifolds and forms, are an $I$-collection.
\begin{defi}\label{defi:part} Let $I$ and $I'$ be finite sets. an $(I,I')$-partition is a partition $\p$ of the disjoint union $I\disj I'$, such that 
\begin{equation}\label{eq:J cap I}
\forall J\in\p:\,\vert J\cap I\vert=1.
\end{equation}
Let $f:I\to\R$, $f':I'\to\R$, and $C\in(0,\infty)$. For every $J\sub I\disj I'$ we define
\begin{equation}\label{eq:Sum J f f' C}\Sum{J,f,f',C}:=-C\sum_{i\in J\cap I}f(i)+\sum_{i'\in J\cap I'}f'(i').
\end{equation}
A \emph{$(f,f',C)$-partition} is an $(I,I')$-partition $\p$ such that
\begin{equation}\label{eq:C sum f i}
\Sum{J,f,f',C}\geq0,\quad\forall J\in\p.
\end{equation}
\end{defi}  
\begin{defi} \label{defi:I+ I- I'}
Let $I^+,I^-,I'$ be finite sets. We denote $I:=I^+\disj I^-$.
\emph{an $\big(I^+,I^-,I'\big)$-partition} is a partition $\p$ of $I\disj I'$ with the following properties:
\begin{enua}
\item\label{defi:I+ I- I':J} There exists a unique element of $\p$ that intersects both $I^+$ and $I^-$ in exactly one point. 
\item\label{defi:I+ I- I':other} All other $J\in\p$ intersect $I$ in exactly one point.
\end{enua}
Let $f^\pm:I^\pm\to\R$, $f':I'\to\R$, and $C\in(0,\infty)$. We denote by $f:=f^+\disj f^-:I\to\R$ the disjoint union of functions.\footnote{This is the function defined by $f(i):=f^\pm(i)$ if $i\in I^\pm$.} A \emph{$(f^+,f^-,f',C)$-partition} is an $(I^+,I^-,I')$-partition satisfying \reff{eq:C sum f i}.
\end{defi}
\begin{rmk}\label{rmk:I+ I- I':card} \textnormal{Every $(I^+,I^-,I')$-partition $\p$ satisfies
\[\vert\p\vert=\vert I\vert-1.\]
}
\end{rmk}
Let $A_0$ be an interval and $I$ a collection of finite sets indexed by $A_0$, i.e., a map from $A_0$ to the class of all finite sets. We denote $I_a:=I(a)$. Let $f=\big(f_a:I_a\to\R\big)_{a\in A_0}$ be a collection of functions. We define
\begin{equation}\label{eq:C f0}C^f_0:=\sup\big\{C\in(0,\infty)\,\big|\,\exists a,a'\in A_0:\,a>a',\,\exists\big(f_a,f_{a'},C\big)\text{-partition}\big\},
\end{equation}
\begin{eqnarray}\label{eq:C f1}C^f_1&:=\sup\big\{C\in(0,\infty)\,\big|&\exists a,a'\in A_0\cap(0,\infty):\,a<a',\\
\nn&&\exists\big(f_a,f_{-a},f_{a'},C\big)\text{-partition}\big\}.
\end{eqnarray}
Here we use the convention that $\sup\emptyset:=0$.
\begin{defi}[$I$-collection]\label{defi:I collection} We call $f$ an $I$-collection iff the following holds:
\begin{align}
\label{defi:I collection:sup >}C^f_0&<1,\\ 
\label{defi:I collection:sup <}C^f_1&<1.
\end{align}
\end{defi}

\begin{Rmk} The condition of being an $I$-collection is invariant under rescaling by some positive constant.
\end{Rmk}

\subsection{Cardinality of the set of capacities in a more general setting, sufficient conditions for being an $I$-collection, proof of Theorem \ref{thm:card cap gen}(\ref{thm:card cap gen:caps},\ref{thm:card cap gen:ncaps})}

Theorem \ref{thm:card cap gen}(\ref{thm:card cap gen:caps},\ref{thm:card cap gen:ncaps}) is a special case of the following more general result. We call a $k$-form $\om$ on a $kn$-manifold \emph{maxipotent} iff $\om^{\wedge n}=\om\wedge\cdots\wedge\om$ does not vanish anywhere.%
\footnote{See Remark \ref{rmk:nondeg} for the relation between maxipotency and nondegeneracy.
} 
In this case we denote
\[\oo_\om:=\textrm{orientation on $M$ induced by }\om^{\wedge n}.\]
Recall that $B,Z$ denote the unit ball and the standard symplectic cylinder, $\omst$ the standard symplectic form, $c_{M,\om}$ the domain-embedding capacity for $(M,\om)$ as in \reff{eq:c M om}, and $w$ the Gromov width as in \reff{eq:w}.
\begin{thm}[cardinality of the set of (normalized) capacities, more general setting] \label{thm:card cap hel} The following holds:
\begin{enui}
\item\label{thm:card cap hel:caps} Let $k,n\in\{2,3,\ldots\}$ with $k$ even, and $\CCC=(\OO,\M)$ be a $(kn,k)$-form category. Then the cardinality of $\caps(\CCC)$ equals $\beth_2$, provided that there exist an interval $A_0$ around 0 of positive length, and a collection $(M_a,\om_a)_{a\in A_0}$ of objects of $\OMEGA{kn}{k}$, such that for every $a\in A_0$, $M_a$ is nonempty, compact, and 1-connected,\footnote{This means connected and simply connected.} $\om_a$ is maxipotent and exact, and the following holds:
\begin{enua}
\item\label{thm:card cap hel:union} $(W_a,\eta_a):=\left(M_a\disj M_{-a},\om_a\disj\om_{-a}\right)\in\OO$, for every $a\in A_0\cap(0,\infty)$.
\item\label{thm:card cap hel:helicity} We denote by $I_a$ the set of connected components of $\partial M_a$, and $I:=(I_a)_{a\in A_0}$. The collection of boundary helicities $f:=\left(h_{M_a,\oo_{\om_a},\om_a}\right)_{a\in A_0}$ is an $I$-collection.
\end{enua}
\item\label{thm:card cap hel:ncaps} Let $n\in\{2,3,\ldots,\}$ and $\CCC=(\OO,\M)$ be a $(2n,2)$-form category that contains the objects $B$ and $Z$. Then the cardinality of $\nc(\CCC)$ equals $\beth_2$, provided that there exist $A_0$ and $(M_a,\om_a)_{a\in A_0}$ as in \reff{thm:card cap hel:caps}, such that also the following holds:
\begin{enua}
\item\label{thm:card cap hel:w} $\sup_{a\in A_0}w(M_a,\om_a)<1$
\item\label{thm:card cap hel:Z} $\sup_{a\in A_0}c_{M_a,\om_a}(Z,\omst)\leq\pi$
\end{enua}
\end{enui}
\end{thm}

We will prove this theorem in Section \ref{sec:proof:thm:card cap hel}. The idea of the proof is to consider the family of capacities
\[c_A:=\sup_{a\in A}c_{W_a,\eta_a},\quad A\in\PP\big(A_0\cap(0,\infty)\big).\]
Hypothesis \reff{thm:card cap hel:helicity} implies that there exists $c_0<1$ such that for all $a\neq a'\in(0,\infty)$ and $c\geq c_0$, the pair $(W_a,c\eta_a)$ does not embed into $(W_{a'},\eta_{a'})$. See the explanations below. It follows that
\[\sup\left\lbrace c_A(W_a,\eta_a)\,\big|\,a\in A_0\cap(0,\infty)\wo A\right\rbrace<1,\quad\forall A.\]
Since also $c_A(W_a,\eta_a)=1$, for every $a\in A$, it follows that
\[c_A\neq c_{A'},\quad\textrm{if }A\neq A'.\]
Since the cardinality of $\PP((0,\infty))$ equals $\beth_2$, it follows that the cardinality of $\caps(\CCC)$ is at least $\beth_2$. On the other hand, we denote by $S$ the set of isomorphism classes of symplectic manifolds. This set has cardinality $\beth_1$. Since $\caps(\CCC)$ can be viewed as a subset of $[0,\infty]^S$, it has cardinality at most $\beth_2$, hence equal to $\beth_2$. 

A refined version of this argument shows that $\vert\nc(\CCC)\vert=\beth_2$. For this we normalize each capacity $c_A$, by replacing it by the maximum of $c_A$ and the Gromov width. Hypothesis \reff{thm:card cap hel:w} guarantees that the modified capacities are still all different from each other. Hypothesis \reff{thm:card cap hel:Z} guarantees that they are normalized.\\

\label{helicity argument} To understand the reason why no big multiple of $(W_a,\eta_a)$ embeds into $(W_{a'},\eta_{a'})$, consider the case in which each $M_a$ is a spherical shell in a symplectic vector space, with inner radius 1 and outer radius $r+a$ for some fixed $r>1$. Assume that $(M_a,c\om_a)$ embeds into $(M_{a'},\om_{a'})$ in such a way that the image of the inner boundary sphere of $M_a$ wraps around the inner boundary sphere of $M_{a'}$. By Corollary \ref{cor:vol hel} (Stokes' Theorem for helicity) and Remark \ref{rmk:or} the difference of the helicities of these spheres equals the enclosed volume on the right hand side. Since this volume is nonnegative, it follows that $c\geq1$. Using our hypothesis \reff{thm:card cap hel:helicity} that the collection of boundary helicities is an $I$-collection, it follows that $a\leq a'$.

It follows that if $a>a'$ then no multiple of $W_a$ (symplectically) embeds into $W_{a'}$ in such a way that the inner boundary sphere of $M_a$ wraps around one of the two inner boundary spheres of $W_{a'}$. Figure \ref{fig:big wrap} on p.~\pageref{fig:big wrap} illustrates this. In contrast with this, Figure \ref{fig:not wrap} shows a possible embedding. In this case our helicity hypothesis \reff{thm:card cap hel:helicity} implies that the rescaling factor is small.

If $a<a'$ then $M_a$ embeds into $M_{a'}$ (without rescaling). However, there is not enough space left for $M_{-a}$. See Figure \ref{fig:small wrap}.\\ 

In the proof of Theorem \ref{thm:card cap gen}\reff{thm:card cap gen:caps} we will use the following sufficient criterion for condition \reff{thm:card cap hel:helicity} of Theorem \ref{thm:card cap hel}. For every finite set $S$ and every function $f:S\to\R$ we denote
\begin{equation}\label{eq:sum f}\sum f:=\sum_{s\in S}f(s).\end{equation}
Let $A_0$ be an interval, $I:=(I_a)_{a\in A_0}$ a collection of finite sets, and $f=\big(f_a:I_a\to\R\big)_{a\in A_0}$ a collection of functions. We define the disjoint unions of $I$ and $f$ to be
\begin{eqnarray*}&\Disj I:=\Disj_{a\in A_0}I_a:=\big\{(a,i)\,\big|\,a\in A_0,\,i\in I_a\big\},&\\
&\Disj f:\Disj I\to\R,\quad\Disj f(a,i):=f_a(i).&
\end{eqnarray*}

\begin{figure}[H]
\centering
\leavevmode
\scalebox{\factor}{\epsfbox{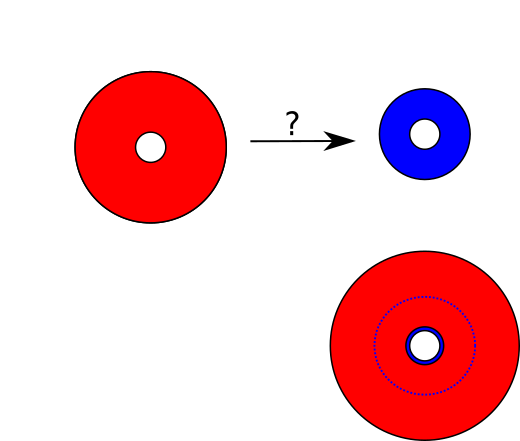}}
\caption{If $a>a'$ then no multiple of the red spherical shell $M_a$ (symplectically) embeds into the blue shell $M_{a'}$ in such a way that the inner boundary sphere of the red shell wraps around the inner boundary sphere of the blue shell, since our helicity hypothesis \reff{thm:card cap hel:helicity} forces the rescaling factor to be at least 1.}
\label{fig:big wrap}
\end{figure}
\begin{figure}[H]
\centering
\leavevmode
\scalebox{\factor}{\epsfbox{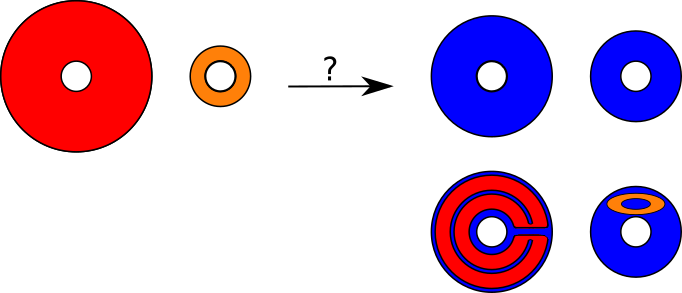}}
\caption{A possible embedding of $(W_a,c\eta_a)$ into $(W_{a'},\eta_{a'})$ in the case $a>a'$. The constant $c$ needs to be small (even if $a$ is close to $a'$), since the volume of the hole enclosed by the image of $M_a$ equals minus $c$ times the helicity of the inner boundary sphere of $M_a$. Here we use again our helicity hypothesis \reff{thm:card cap hel:helicity}.}
\label{fig:not wrap}
\end{figure}
\begin{figure}[H]
\centering
\leavevmode\scalebox{\factor}{\epsfbox{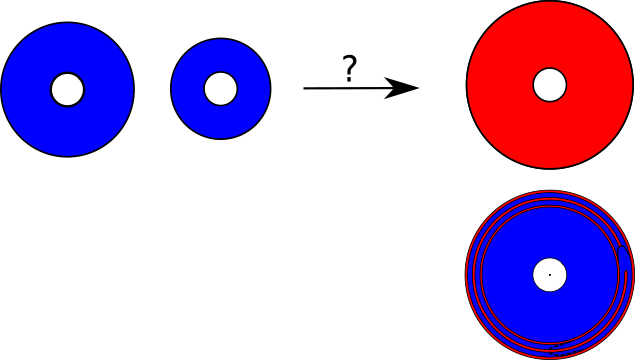}}
\caption{An attempt for an embedding of $W_a$ into $M_{a'}$ in the case $a<a'$ (without rescaling). The image of $M_{-a}$ overlaps itself, since there is not enough space left in $M_{a'}$.}
\label{fig:small wrap}
\end{figure}

\begin{prop}[sufficient conditions for being an $I$-collection] \label{prop:I collection} The collection $f$ is an $I$-collection if there exists $\ell\in\N_0$, such that the following holds:
\begin{enui}
\item For all $a\in A_0$ we have
\begin{eqnarray}
\label{eq:Ia}&|I_a|=\ell,&\\
\label{eq:f -1}&f_a\geq-1,&\\
\label{eq:-1}&f_a^{-1}(-1)\neq\emptyset,&\\
\label{eq:f-1 0 infty}&\vert f_a^{-1}((0,\infty))\vert=1,&\\
\label{eq:sum a leq1}&\displaystyle\sum f_a\leq1.&
\end{eqnarray}
\item For all $a,a'\in A_0$ we have
\begin{equation}\label{eq:sum a sum a'}\sum f_a>\sum f_{a'},\quad\textrm{if }a>a'.
\end{equation}
\item We have
\begin{equation}
\label{eq:sup}\sup\left(\im\left(\Disj f\right)\cap(-\infty,0]\right)<-1+\inf_{a\in A_0}\sum f_a.
\end{equation}
If $\ell\geq4$ then we have
\begin{equation}
\label{eq:sup<2inf+1}\sup\Disj f<2\inf\left(\im\left(\Disj f\right)\cap(0,\infty)\right)+1.
\end{equation}
\end{enui}
\end{prop}
\journal{\begin{proof} This follows from an elementary argument. See the arXiv-version of this article.
\end{proof}
}
\arxiv{We will prove this proposition in Section \ref{sec:proof:prop:I collection}.}

\begin{Rmk} The conditions (\ref{eq:-1},\ref{eq:f-1 0 infty}) imply that $\ell\geq2$.
\end{Rmk}

\begin{proof}[Proof of Theorem \ref{thm:card cap gen}(\ref{thm:card cap gen:caps},\ref{thm:card cap gen:ncaps})] \textbf{\reff{thm:card cap gen:caps}:} We choose $V,\Om,K,r$ as in the hypothesis. We define
\begin{equation}\label{eq:om Om}\om:=\Om^{\oplus n}.\end{equation}
Since by hypothesis, $k$ is even and $\Om$ is a volume form, the form $\om$ is maxipotent, i.e., $\om^{\wedge n}$ is a volume form. We denote by $\oo$ the orientation on $V^n$ induced by this form. Since by hypothesis, $K$ is nonempty and strictly starshaped around 0, its interior contains 0. It follows that
\begin{equation}\label{eq:C int}C:=\int_{K,\oo}\om^{\wedge n}>0.\footnote{Here we view $V^n$ as a manifold and $\om$ as a differential form on it.}\end{equation}
By hypothesis, we have
\begin{equation}\label{eq:a1}a_1:=\min\big\{r-1,\sqrt[kn]2-r\big\}>0.\end{equation}
We choose $a_0\in(0,a_1)$ and define $A_0:=[-a_0,a_0]$. For every $a\in A_0$ we define
\begin{eqnarray}
\label{eq:Ma}&M_a:=(r+a)K\wo\INT K,&\\
\label{eq:om a}&\om_a:=C^{-\frac1n}\om\big|M_a,&\\
\label{eq:I a}&I_a:=\big\{\textrm{connected component of }\bd M_a\big\},&\\
\nn&I:=(I_a)_{a\in A_0}.&
\end{eqnarray}
The form $\om_a$ is well-defined, since $C>0$. We check the hypotheses of Theorem \ref{thm:card cap hel}\reff{thm:card cap hel:caps}. Let $a\in A_0$. The set $M_a$ is compact. Since $K$ is strictly starshaped around 0, $M_a$ is a smooth submanifold of $V^n$ that continuously deformation retracts onto $\bd K$. The manifold $\bd K$ is homeomorphic to the sphere $S^{kn-1}_1$. Since by hypothesis $k,n\geq2$, this sphere is 1-connected. Hence the same holds for $M_a$. The form $\Om$ is exact. Hence the same holds for $\om$ and thus for $\om_a$.

Condition \reff{thm:card cap hel:union} is satisfied by our hypothesis and the rescaling property for a $(kn,k)$-form category. We show that the collection of boundary helicities
\begin{equation}\label{eq:f a}f:=\left(f_a:=h_{M_a,\oo_{\om_a},\om_a}\right)_{a\in A_0}\end{equation}
satisfies \reff{thm:card cap hel:helicity}. We check the hypotheses of Proposition \ref{prop:I collection}. Let $s\in(0,\infty)$. We denote by $\oo^s$ the orientation on $\bd(sK)$ induced by $\oo$ and $sK$. By Lemma \ref{le:vol hel} we have 
\begin{equation}\label{eq:h dd sK}h\left(\bd(sK),\oo^s,\om_{\bd(sK)}\right)=\int_{sK,\oo}\om^{\wedge n}=Cs^{kn}.\end{equation}
For every connected component $i$ of $\bd M_a$ we denote by $\oo_i$ the orientation of $i$ induced by $\oo,M_a$. Using (\ref{eq:h dd sK},\ref{eq:om a}) and Remarks \ref{rmk:resc},\ref{rmk:or}, we obtain
\begin{equation}\label{eq:h i oo i om a i}h\left(i,\oo_i,(\om_a)_i\right)=\left\{\begin{array}{ll}
(r+a)^{kn},&\textrm{for }i=\bd\big((r+a)K\big),\\
-1,&\textrm{for }i=\bd K.
\end{array}\right.\end{equation}
Here we used that the orientation of $\bd K$ induced by $\oo$ and $M_a$ is the opposite of $\oo^1$. It follows that
\begin{equation}\label{eq:sum i Ia}\sum f_a:=\sum_{i\in I_a}f_a(i)=-1+(r+a)^{kn}\in\big[-1+(r-a_0)^{kn},-1+(r+a_0)^{kn}\big],\,\forall a\in A_0.\end{equation}
Since $a_0<a_1\leq r-1$, we have $-1+(r-a_0)^{kn}>0$. Hence by \reff{eq:sum i Ia}, we have $\inf_{a\in A_0}\sum f_a>0$. Using \reff{eq:h i oo i om a i}, it follows that condition \reff{eq:sup} is satisfied. 

Since $a_0<a_1\leq\sqrt[kn]2-r$, we have $-1+(r+a_0)^{kn}<1$. Using \reff{eq:sum i Ia}, it follows that $\sup_{a\in A_0}\sum f_a<1$. Hence inequality \reff{eq:sum a leq1} is satisfied. The collection $f$ also satisfies the other hypotheses of Proposition \ref{prop:I collection}. Applying this proposition, it follows that $f$ is an $I$-collection. Hence condition \reff{thm:card cap hel:helicity} is satisfied.

Therefore, all hypotheses of Theorem \ref{thm:card cap hel}\reff{thm:card cap hel:caps} are satisfied. Applying this theorem, it follows that the cardinality of $\caps(\CCC)$ equals $\beth_2$. This proves Theorem \ref{thm:card cap gen}\reff{thm:card cap gen:caps}.\\

To prove \reff{thm:card cap gen:ncaps}, assume that the hypotheses of this part of the theorem are satisfied. We choose $r\in\left(1,\sqrt[2n]2\right)$ satisfying \eqref{eq:sh}. We define $V:=\R^2$, $\Om$ to be the standard area form on $\R^2$, $K:=\BAR B^{2n}_1$, and $a_1$ as in \reff{eq:a1}. We choose $a_0\in(0,a_1)$, and define $A_0:=[-a_0,a_0]$ and $(M_a,\om_a)$ as in (\ref{eq:Ma},\ref{eq:om a}). The tripel $(V,\Om,K)$ satisfies the conditions of part \reff{thm:card cap gen:caps} of Theorem \ref{thm:card cap gen}. Hence by what we proved above, the collection $(M_a,\om_a)$, $a\in A_0$, satisfies the conditions of Theorem \ref{thm:card cap hel}\reff{thm:card cap hel:caps}. Let $a\in A_0$. By \eqref{eq:om Om} we have $\om=\omst$. Using \eqref{eq:C int}, it follows that
\[C=\int_{K=\BAR B^{2n}_1}\omst^{\wedge n}=\pi^n,\]
and therefore,
\begin{equation}\label{eq:Ma om a}(M_a,\om_a)=\left(\sh_{r+a},\frac1\pi\omst\big|M_a\right).\end{equation}
We check condition \reff{thm:card cap hel:w} of Theorem \ref{thm:card cap hel}. Let $a\in A_0$. It follows from \eqref{eq:Ma om a} that the symplectic volume of $(M_a,\om_a)$ is $\frac{-1+(r+a)^{2n}}{\pi^n}$ times the volume of the unit ball. Therefore, we have
\[w(M_a,\om_a)\leq\sqrt[n]{-1+(r+a)^{2n}}.\]
Using the inequalities $a\leq a_0<a_1\leq\sqrt[2n]2-r$, it follows that condition \reff{thm:card cap hel:w} is satisfied.

We check \reff{thm:card cap hel:Z}. Let $a\in A_0$. Since $r+a\geq r-a_0>r-a_1\geq1$, the shell $\sh_{r+a}$ contains the sphere $S^{2n-1}_{r+a}$. Using $r+a>1$, $\om_a=\frac1\pi\omst$, and $n\geq2$, it follows from \cite[Corollary 5, p. 8]{SZ} (spherical nonsqueezing) that $(M_a,b\om_a)$ does not symplectically embed into $Z$ for any $b\geq\pi$. Hence \reff{thm:card cap hel:Z} holds.

Therefore, all hypotheses of Theorem \ref{thm:card cap hel}\reff{thm:card cap hel:ncaps} are satisfied. Applying this part of the theorem, it follows that the cardinality of $\nc(\CCC)$ equals $\beth_2$. This proves Theorem \ref{thm:card cap gen}\reff{thm:card cap gen:ncaps}.
\end{proof}
\section{Proof of Theorem \ref{thm:card cap hel} (cardinality of the set of capacities, more general setting)} \label{sec:proof:thm:card cap hel}

As mentioned, the idea of proof of Theorem \ref{thm:card cap hel} is that our helicity hypothesis \reff{thm:card cap hel:helicity} and Stokes' Theorem for helicity imply that for $a\neq a'$ only small multiples of $(W_a,\eta_a)$ embed into $(W_{a'},\eta_{a'})$. The idea behind this is that every embedding $\phi$ of $M_a$ into $M_{a'}$ gives rise to a partition of the disjoint union of the sets of connected components of $\bd M_a$ and $\bd M_{a'}$. The elements of this partition consist of components that lie in the same connected component of the complement of $\phi(\Int M)$. Here $\Int M$ denotes the interior of $M$ as a manifold with boundary, and we identify each component of $\bd M_a$ with its image under $\phi$.

Stokes' Theorem for helicity implies that the inequality \eqref{eq:C sum f i} is satisfied. Together with a similar argument in which we consider embeddings of $W_a$ into $M_{a'}$, it follows that the partition satisfies the conditions of Definitions \ref{defi:part},\ref{defi:I+ I- I'}. Combining this with our helicity hypothesis \reff{thm:card cap hel:helicity}, it follows that indeed only small multiples of $W_a$ embed into $W_{a'}$.\\

Lemmata \ref{le:part} and \ref{le:hel} below will be used to make this argument precise. To formulate the first lemma, we need the following.

\begin{rmk}[pullback of relation]\label{rmk:rel} \textnormal{Let $S',S$ be sets, $R$ a relation on $S$, and $f:S'\to S$ a map. Denoting by $\x$ the Cartesian product of maps, the set
\[R':=f^*R:=(f\x f)^{-1}(R)\]
is a relation on $S'$. If $R$ is reflexive/ symmetric/ transitive, then the same holds for $R'$.}
\end{rmk}
Let $X$ be a topological space. We define
\begin{equation}\label{eq:conn X}\conn{X}:=\big\{\textrm{path-connected subset of }X\big\}\end{equation}
and the relation $\sim_X$ on $\conn{X}$ by
\begin{equation}\label{eq:sim X}A\sim_XB:\iff\exists\textrm{ continuous path starting in $A$ and ending in $B$.}\end{equation}
This is an equivalence relation.

Let $M$ and $M'$ be topological manifolds of the same dimension, and $\phi:M\to M'$ a topological embedding, i.e., a homeomorphism onto its image. We denote by $\Int(M)$ and $\bd M$ the interior and the boundary of $M$ as a manifold with boundary. We denote
\begin{eqnarray}\label{eq:I M}&I:=I_M:=\big\{\textrm{connected component of }\bd M\big\},\quad I':=I_{M'}&\\
&P:=M'\wo\phi(\Int(M)).&
\end{eqnarray}
We define
\begin{eqnarray}\nn&\Phi:\PP(M)\to\PP(M'),\quad\Phi(A):=\textrm{ image of $A$ under }\phi,&\\
\nn&\Psi:I\disj I'\to\PP(P),\quad\Psi:=\Phi\textrm{ on }I,\quad\Psi:=\id\textrm{ on }I',&\\
\nn&\sim^\phi:=\Psi^*\sim_{P},&\\
\nn&\p^\phi:=\textrm{partition of $I\disj I'$ associated with }\sim^\phi.
\end{eqnarray}
\begin{rmk}[partition induced by embedding]\label{rmk:part} \textnormal{For every path-component $P_0$ of $P$ we define
\begin{align}\label{eq:J A'}J^\phi(P_0)&:=\Psi^{-1}(\PP(P_0))\\
\nn&=\big\{i\in I\,\big|\,\Phi(i)\in\PP(P_0)\big\}\disj(I'\cap\PP(P_0)).
\end{align}
The map
\[J^\phi:\big\{\textrm{path-component $P_0$ of }P:\,J^\phi(P_0)\neq\emptyset\big\}\to\p^\phi\]
is well-defined and a bijection.}
\end{rmk}
For every field $F$ and $i\in\N_0$ we denote by $H_i(M;F)$ the degree $i$ singular homology of $M$ with coefficients in $F$.
\begin{lemma}[partition associated with an embedding]\label{le:part} Assume that $M,M'$ are compact, $M'$ is connected, $\bd M'\neq\emptyset$, and that there exists a field $F$, for which $H_1(M';F)$ vanishes. Then the following holds:
\begin{enui}
\item\label{le:part:conn} If $M$ is nonempty and connected then $\p^\phi$ is an $(I_M,I_{M'})$-partition.
\item\label{le:part:two} If $M$ consists of precisely two connected components $M^+$ and $M^-$ then $\p^\phi$ is an $\big(I_{M^+},I_{M^-},I_{M'}\big)$-partition.
\end{enui}
\end{lemma} 
Recall that the first statement means that condition \eqref{eq:J cap I} is satisfied, i.e., \mbox{$|J\cap I_M|=1$} for every $J\in\p^\phi$. The idea of proof of the inequality $\leq1$ is the following. Each $J$ corresponds to a path-component $P_0$ of the complement of $\phi(\Int M)$. Suppose that there exists $J$ that intersects $I_M$ in at least two points $i_0,i_1$ (= components of $\bd M$). Then there is a path in $P_0$ joining $\phi(i_0)$ and $\phi(i_1)$. By connecting this path with a path in $\phi(M)$ with the same endpoints, we obtain a loop in $M'$ that intersects $i_0$ and $i_1$ in one point each. See Figure \ref{fig:torus}. %
\begin{figure}[H]
\centering
\leavevmode\scalebox{.75}{\epsfbox{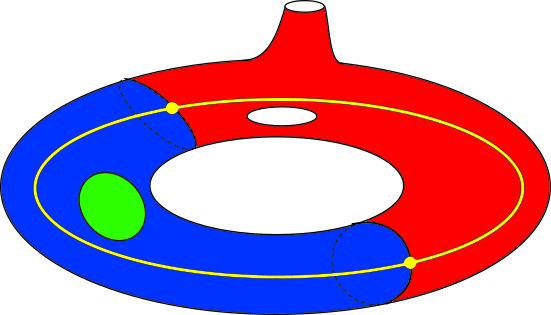}}
\caption{The blue region is the image of $M$ under $\phi$, and the red and green regions are the path-components of the complement of $\phi(\Int M)$. The red region contains the images of two connected components $i_0,i_1$ of the boundary of $M$. The yellow loop intersects these images in one point each.}
\label{fig:torus}
\end{figure}
Hence the algebraic intersection number of this loop with $i_0$ equals 1. In particular, it represents a nonzero first homology class. Hence the hypothesis that the first homology of $M'$ vanishes, is violated. It follows that $|J\cap I_M|\leq1$.

In order to make this argument precise one needs to ensure that the algebraic intersection number equals the ``na\"ive intersection number''. For simplicity, we therefore use an alternative method of proof, which is based on a certain Mayer-Vietoris sequence for singular homology. We need the following.
\begin{rmk}[embedding is open, boundary]\label{rmk:im int bd}\textnormal{We denote by $\btop{S}{X}$ the boundary of a subset $S$ of a topological space $X$. Let $M,M'$ be topological manifolds of the same dimension $n$, and $\phi:M\to M'$ an injective continuous map. By invariance of the domain, in every pair of charts for $\Int M$ and $M'$, the map $\phi$ sends every open subset of $\R^n$ to an open subset of $\R^n$. It follows that the set $\phi(\Int M)$ is open in $M'$. This implies that
\[\phi(\bd M)\sub\btop{\phi(\Int M)}{M'},\]
and if $M$ is compact, then equality holds.
}

\textnormal{Suppose now that $M$ is nonempty and compact, $\bd M=\emptyset$, and $M'$ is connected. Then $M'$ has no boundary, either. To see this, observe that $\phi(M)$ is compact, hence closed in $M'$. Since $M=\Int M$, as mentioned above, $\phi(M)$ is also open. Since $M'$ is connected, it follows that $\phi(M)=M'$. Since in every pair of charts for $M$ and $M'$, $\phi$ sends every open subset of $\R^n$ to an open subset of $\R^n$, it follows that $\bd M'=\emptyset$. 
}
\end{rmk}
\begin{proof}[Proof of Lemma \ref{le:part}] Assume that $M,M'$ are compact, $M\neq\emptyset$, $M'$ is connected, and $\bd M'\neq\emptyset$. We denote
\[I:=I_M,\quad I':=I_{M'},\quad P:=M'\wo\phi(\Int(M)),\]
and by $k$ the number of connected components of $M$.
\begin{claim}\label{claim:p phi} We have 
\begin{equation}\label{eq:p phi}|\p^\phi|=|I|+1-k.\end{equation}
\end{claim}
\begin{proof}[Proof of Claim \ref{claim:p phi}] Let $P_0$ be a path-component of $P$.
\begin{claim}\label{claim:conn comp phi bd M} $P_0$ intersects $\phi(\bd M)$.
\end{claim}
\begin{proof}[Proof of Claim \ref{claim:conn comp phi bd M}] By Remark \ref{rmk:im int bd} we have $\bd M\neq\emptyset$. Since by hypothesis, $M'$ is connected, there exists a continuous path $x':[0,1]\to M'$ that starts in $P_0$ and ends at $\phi(\bd M)$. Since $M$ is compact, the same holds for $\bd M$, and hence for $\phi(\bd M)$. Hence the minimum
\[t_0:=\min\big\{t\in[0,1]\,\big|\,x'(t)\in\phi(\bd M)\big\}\]
exists. By Remark \ref{rmk:im int bd} the set $\phi(\Int M)$ is open in $M'$. It follows that $x'(t_0)\not\in\phi(\Int M)$, and hence $x'([0,t_0])\sub P=M'\wo\phi(\Int M)$. (In the case $t_0=0$ this holds, since $x'(0)\in P_0\sub P$.) It follows that $x'(t_0)\in P_0$. Since also $x'(t_0)\in\phi(\bd M)$, it follows that $P_0\cap\phi(\bd M)\neq\emptyset$. This proves Claim \ref{claim:conn comp phi bd M}.
\end{proof}
Claim \ref{claim:conn comp phi bd M} implies that the set $J^\phi(P_0)$ (defined as in \reff{eq:J A'}) is nonempty. Hence by Remark \ref{rmk:part} we have
\begin{equation}\label{eq:conn comp}\big|\big\{\textrm{path-component of }P\big\}\big|=|\p^\phi|.
\end{equation}
By M.~Brown's Collar Neighbourhood Theorem \cite{Br} (see also \cite[Theorem, p.~180]{Co}) there exists an open subset $V$ of $M$ and a (strong) deformation retraction $h$ of $V$ onto $\bd M$. We define
\[A:=\phi(M),\quad B:=M'\wo\phi(M\wo V).\]
Extending $\phi\circ h_t\circ\phi^{-1}:\phi(V)\to\phi(V)$ by the identity, we obtain a map $h':[0,1]\x B\to B$. Since by Remark \ref{rmk:im int bd}, the restriction of $\phi$ to $\Int M$ is open, the map $h'$ is continuous, and therefore a deformation retraction of $B$ onto $P$.

We choose a field $F$ as in the hypothesis, and denote by $H_i$ singular homology in degree $i$ with coefficients in $F$. Since $P$ is a deformation retract of $B$, these spaces have isomorphic $H_0$. Combining this with \reff{eq:conn comp}, it follows that
\begin{align}\nn|\p^\phi|&=\left|\big\{\textrm{path-component of }P\big\}\right|\\
\nn&=\dim H_0(P)\\
\label{eq:|p phi|}&=\dim H_0(B).
\end{align}
The interiors of $A$ and $B$ cover $M'$. Therefore, the Mayer-Vietoris Theorem implies that there is an exact sequence
\[\ldots\to H_1(M')\to H_0(A\cap B)\to H_0(A)\oplus H_0(B)\to H_0(M')\to0.\]
Since by hypothesis, $H_1(M')=0$, it follows that
\begin{equation}\label{eq:dim H0 B'}\dim H_0(B)=\dim H_0(A\cap B)+\dim H_0(M')-\dim H_0(A).\end{equation}
Since $A\cap B=\phi(V)$ and $\phi$ is a homeomorphism onto its image, we have $H_0(A\cap B)\iso H_0(V)$. Since $V$ deformation retracts onto $\bd M$, we have $H_0(V)\iso H_0(\bd M)$, hence $H_0(A\cap B)\iso H_0(\bd M)$. Since $\bd M$ is a topological manifold, its path-components are precisely its connected components. Recalling the definition \reff{eq:I M} of $I$, it follows that
\begin{equation}\label{eq:dim H0 dd M}\dim H_0(A\cap B)=|I|.\end{equation}
Since by hypothesis $M'$ is connected, we have
\begin{equation}\label{eq:dim H0 M'}\dim H_0(M')=1.\end{equation}
Since $A:=\phi(M)$, we have $H_0(A)\iso H_0(M)$, and therefore
\[\dim H_0(A)=k.\]
Combining this with (\ref{eq:|p phi|},\ref{eq:dim H0 B'},\ref{eq:dim H0 dd M},\ref{eq:dim H0 M'}), equality \reff{eq:p phi} follows. This proves Claim \ref{claim:p phi}.
\end{proof}
Remark \ref{rmk:part} and Claim \ref{claim:conn comp phi bd M} imply that every element of $\p^\phi$ intersects $I$.

We prove \reff{le:part:conn}. Assume that $M$ is connected. Then by Claim \ref{claim:p phi}, we have $|\p^\phi|=|I|$. It follows that $|J\cap I|=1$, for every $J\in\p^\phi$. Hence $\p^\phi$ is an $(I,I')$-partition. This proves \reff{le:part:conn}.\\

Assume now that $M^\pm$ are as in the hypothesis of \reff{le:part:two}. By Claim \ref{claim:p phi} we have $|\p^\phi|=|I|-1$. Since every element of $\p^\phi$ intersects $I$, it follows that there exists a unique $J_0\in\p^\phi$, such that $|J_0\cap I|=2$, and
\begin{equation}\label{eq:J I 1}|J\cap I|=1,\quad\forall J\in\p^\phi\wo\{J_0\}.\end{equation}
By Remark \ref{rmk:part} there exists a unique path-component $P_0$ of $P$, such that $J_0=J^\phi(P_0)$.
\begin{claim}\label{claim:J0 I+} We have
\[J_0\cap I^-\neq\emptyset\neq J_0\cap I^+.\]
\end{claim}
\begin{proof}[Proof of Claim \ref{claim:J0 I+}] We denote by $P^+_0$ the path-component of $M'\wo\phi(\Int(M^+))$ containing $P_0$. Assume by contradiction that $P^+_0\cap\phi(M^-)=\emptyset$. Then we have
\[P^+_0=P_0,\quad J^{\phi|M^+}(P^+_0)=J^\phi(P_0)=J_0,\quad J_0\cap I=J_0\cap I^+.\]
Since $|J_0\cap I|=2$, we obtain a contradiction with \reff{le:part:conn}, with $I,\phi$ replaced by $I^+,\phi|M^+$. Hence we have
\[P^+_0\cap\phi(M^-)\neq\emptyset.\]
It follows that there exists a continuous path $x':[0,1]\to M'\wo\phi(\Int(M^+))$ that starts at $P_0$ and ends at $\phi(M^-)$. Since $M$ is compact, the same holds for $\phi(M^-)$. Hence the minimum
\[t_0:=\min\big\{t\in[0,1]\,\big|\,x'(t)\in\phi(M^-)\big\}\]
exists. By Remark \ref{rmk:im int bd} the set $\phi(\Int M^-)$ is open. It follows that $x'(t_0)\not\in\phi(\Int M^-)$, hence $x'([0,t_0])\sub P$, and therefore
\begin{equation}\label{eq:x' t0 }x'(t_0)\in P_0.\end{equation}
On the other hand $x'(t_0)\in\phi(M^-)\sub\BAR{\phi(\Int M^-)}$, and therefore
\[x'(t_0)\in\btop{\phi(\Int M^-)}{M'}=\phi(\bd M^-).\]
Here we used Remark \ref{rmk:im int bd}. Combining this with \eqref{eq:x' t0 }, it follows that $P_0\cap\phi(\bd M^-)\neq\emptyset$, and therefore $J_0\cap I^-\neq\emptyset$. 

An analogous argument shows that $J_0\cap I^+\neq\emptyset$. This proves Claim \ref{claim:J0 I+}.
\end{proof}
By Claim \ref{claim:J0 I+} and \eqref{eq:J I 1} $\p^\phi$ is an $\big(I^+,I^-,I'\big)$-partition. This proves \reff{le:part:two} and completes the proof of Lemma \ref{le:part}.
\end{proof}

The second ingredient of the proof of Theorem \ref{thm:card cap hel} is the following. Let $k,n\in\N_0$ with $n\geq2$, $M,M'$ be compact (smooth) manifolds of dimension $kn$, $\om,\om'$ exact maxipotent $k$-forms on $M,M'$, $c\in(0,\infty)$, and $\phi:M\to M'$ a (smooth) orientation preserving embedding that intertwines $c\om$ and $\om'$. We denote by $\oo,\oo'$ the orientations of $M,M'$ induced by $\om,\om'$. Recall Definitions \ref{defi:hel},\ref{defi:boundary hel} of (boundary) helicity.
\begin{lemma}[helicity inequality] \label{le:hel} Condition \reff{eq:C sum f i} holds with $\p=\p^\phi$, $f=h_{M,\oo,\om}$, $f'=h_{M',\oo',\om'}$, and $C:=c^n$.
\end{lemma}
The reason for this is that the left hand side of \reff{eq:C sum f i} is the volume of the path-component of the complement of $\phi(\Int M)$, determined by $J$. To make this precise, we need the following.
\begin{rmk}\label{rmk:conn comp} \textnormal{Let $X$ and $X'$ be topological spaces and $f:X\to X'$ be continuous. Recall the definitions (\ref{eq:conn X},\ref{eq:sim X}) of $\conn{X}$ and $\sim_X$. 
\begin{enui}\item\label{rmk:conn comp:well} The map
\[f_*:\conn{X}\to\conn{X'},\quad f_*(A):=f(A),\]
is well-defined. Furthermore, we have
\[f_*\x f_*(\sim_X)\sub\sim_{X'}.\]
\item\label{rmk:conn comp:reverse} Assume that $X=X'$ and for every $x\in X$ there exists a continuous path from $x$ to $f(x)$. Then for every pair $A,B\in\conn{X}$ we have
\[f_*(A)\sim_{f(X)}f_*(B)\then A\sim_XB.\]
This follows from transitivity of $\sim_X$.
\end{enui}
}
\end{rmk}
\begin{proof}[Proof of Lemma \ref{le:hel}]\setcounter{claim}{0} Let $M,\oo,\om,M',\oo',\om',c,\phi$ be as in the hypothesis. We define $I=I_M,I'=I_{M'}$ as in \reff{eq:I M}. Consider first the \textbf{case} in which
\begin{equation}\label{eq:phi dd M}\phi(\bd M)\cap\bd M'=\emptyset.\end{equation}
Then the set
\[P:=M'\wo\phi(\Int M)\]
is a smooth submanifold of $M'$. Let $i\in I$. We denote $\hhat i:=\phi(i)$. We define $\oo_N^M$ as in \eqref{eq:oo N}, and abbreviate
\[\oo_i:=\oo_i^M,\quad\oo_{\hhat i}:=(\oo')^{P}_{\hhat i}.\]
Recall that $\BAR\oo$ denotes the orientation opposite to $\oo$. Since $\phi$ intertwines $\oo,\oo'$, and $P,\phi(M)$ lie on opposite sides of $\hhat i$, we have
\begin{equation}\label{eq:oo i'}(\phi|i)_*\BAR{\oo_i}=\BAR{(\phi|i)_*\oo_i}=\oo_{\hhat i}.
\end{equation}
Recall the definition \eqref{eq:om N} of $\om_N$. Since $\phi$ intertwines $c\om,\om'$, we have
\begin{equation}\label{eq:om' i'}(\phi|i)_*c\om_i=\om'_{\hhat i}.
\end{equation}
We have
\begin{align}\nn-c^nh(i,\oo_i,\om_i)&=c^nh(i,\BAR{\oo_i},\om_i)\qquad\textrm{(by Remark \ref{rmk:or})}\\
\nn&=h\big(i,\BAR{\oo_i},c\om_i\big)\qquad\textrm{(by Remark \ref{rmk:resc})}\\
\nn&=h\Big((\phi|i)_*\big(i,\BAR{\oo_i},c\om_i\big)\Big)\qquad\textrm{(by Remark \ref{rmk:hel})}\\
\label{eq:h i'}&=h\big(\hhat i,\oo_{\hhat i},\om'_{\hhat i}\big)\qquad\textrm{(using $\hhat i=\phi(i)$, (\ref{eq:oo i'},\ref{eq:om' i'}))}.
\end{align}
Let $P_0$ be a path-component of $P$. We define $J:=J^\phi(P_0)$ as in \reff{eq:J A'}. Using $h_{M,\oo,\om}(i)=h(i,\oo_i,\om_i)$ and \reff{eq:h i'}, we have
\begin{align*}&\phantom{=}-c^n\sum_{i\in J\cap I}h_{M,\oo,\om}(i)+\sum_{i'\in J\cap I'}h_{M',\oo',\om'}(i')\\
&=\sum_{\hhat i\in I_{P_0}}h_{P_0,\oo'|P_0,\om'|P_0}(\hhat i)\\
&=\int_{P_0,\oo'|P_0}{\om'}^n\qquad\textrm{(using Corollary \ref{cor:vol hel})}\\
&\geq0.
\end{align*}
Hence the statement of Lemma \ref{le:hel} holds in the case \reff{eq:phi dd M}.

Consider now the general situation. Let $(K_i,r_i)_{i\in I}$ be a collection, where for each $i\in I$, $K_i$ is a compact connected neighbourhood of $i$ that is a (smooth) submanifold of $M$ (with boundary), and $r_i:K_i\to i$ is a continuous retraction, such that the sets $K_i$, $i\in I$, are disjoint. We denote by $\INT(K_i)$ the interior of $K_i$ in $M$. We define
\begin{eqnarray*}&\wt M:=M\wo\bigcup_{i\in I}\INT(K_i),&\\
&\wt\phi:=\phi|\wt M,&\\
&\wt I_i:=I_{K_i}\wo\{i\},\,\forall i\in I,\quad\wt I:=I_{\wt M}.
\end{eqnarray*}
We define
\begin{equation}\label{eq:wt}\wt{\phantom{.}}:\PP(I\disj I')\to\PP\big(\wt I\disj I'\big),\quad\wt J:=(J\wo I)\cup\bigcup_{i\in J\cap I}\wt I_i.\end{equation}
The set $\wt M$ is a submanifold of $M$, and
\begin{equation}\label{eq:wt phi}\wt\phi(\bd\wt M)\cap\bd M'=\emptyset.\end{equation}
\begin{claim}\label{claim:P wt phi}
\begin{equation}\label{eq:P wt phi}\p^{\wt\phi}=\wt{\p^\phi}:=\big\{\wt J\,\big|\,J\in\p^\phi\big\}.\end{equation}
\end{claim}
\begin{proof}[Proof of Claim \ref{claim:P wt phi}] We define
\[\wt P:=M'\wo\phi(\Int(\wt M)),\quad r:\wt P\to P,\]
by setting 
\[r:=\left\{\begin{array}{ll}
\phi\circ r_i\circ\phi^{-1}&\textrm{on }\phi(K_i),\textrm{ with }i\in I,\\
r=\id&\textrm{on }M'\wo\phi(M).
\end{array}\right.\]
Since the sets $K_i$ are disjoint, the map $r$ is well-defined. Since by hypothesis, $\phi$ is an embedding between two manifolds of the same dimension, the map $r$ is continuous. Let $i\in I$. Since $K_i$ is path-connected and $r_i$ is a retraction onto the subset $i$ of $K_i$, the hypotheses of Remark \ref{rmk:conn comp}\reff{rmk:conn comp:reverse} are satisfied with $f=r$. Applying this remark, it follows that for every pair $\wt A,\wt B$ of path-connected subsets of $\wt P$ we have
\[\wt A\sim_{\wt P}\wt B\iff r(\wt A)\sim_{r(\wt P)=P}r(\wt B).\]
This implies that if $i_0,i_1\in I$, $\wt i_k\in\wt I_{i_k}$, for $k=0,1$, and $i'_0,i'_1\in I'$ then
\[\wt i_0\sim_{\wt\phi}\wt i_1\iff i_0\sim_\phi i_1,\quad i'_0\sim_{\wt\phi}i'_1\iff i'_0\sim_\phi i'_1,\quad \wt i_0\sim_{\wt\phi}i'_0\iff i_0\sim_\phi i'_0.\]
Equality \reff{eq:P wt phi} follows. This proves Claim \ref{claim:P wt phi}.
\end{proof}
We abbreviate
\[h_M:=h_{M,\oo,\om}.\]
Recall the definition \eqref{eq:Sum J f f' C}. Using \eqref{eq:wt phi}, by what we already proved, condition \reff{eq:C sum f i} holds with $I$ replaced by $\wt I$, $\p:=\p^{\wt\phi}$, $f:=h_{\wt M}$, $f':=h_{M'}$, and $C:=c^n$. Using Claim \ref{claim:P wt phi}, it follows that
\begin{equation}\label{eq:Sum wt J}\Sum{\wt J,h_{\wt M},h_{M'},c^n}\geq0,\quad\forall J\in\p^\phi.
\end{equation}
We denote by $\btop{S}{X}$ the boundary of a subset $S$ of a topological space $X$. For every $i\in I$ Remark \ref{rmk:or} and Lemma \ref{le:vol hel} imply that
\begin{align}\nn h_{\wt M}(\btop{K_i}{M})&=-h_{K_i}(\btop{K_i}{M})\\
\label{eq:h M i}&=h_M(i)-\int_{K_i}\om^{\wedge n},
\end{align}
where the integral is \wrt the orientation $\oo|K_i$. Let $J\in\p^\phi$. Recalling the definition \reff{eq:wt} of $\wt{\phantom{.}}$ and using \reff{eq:h M i}, we have
\[\sum_{\wt i\in\wt J\cap\wt I}h_{\wt M}(\wt i)=\sum_{i\in J\cap I}\left(h_M(i)-\int_{K_i}\om^{\wedge n}\right).\]
Combining this with \reff{eq:Sum wt J} and recalling the definition \eqref{eq:Sum J f f' C}, it follows that
\[\Sum{J,h_M,h_{M'},c^n}\geq-c^n\sum_{i\in J\cap I}\int_{K_i}\om^{\wedge n}.\]
Since this holds for every choice of $(K_i)_{i\in I}$, it follows that $\Sum{J,h_M,h_{M'},c^n}\geq0$. Hence condition \reff{eq:C sum f i} holds with $\p:=\p^\phi$, $f:=h_M$, $f':=h_{M'}$, and $C:=c^n$. This proves Lemma \ref{le:hel}.
\end{proof}
\begin{Rmk}[helicity inequality] \textnormal{Under the hypotheses of this lemma, the set $M'\wo\phi(\Int(M))$ need not be a submanifold of $M'$, since $\phi(\bd M)$ may intersect $\bd M'$. This is the reason for the construction of $\wt M$ in the proof of this lemma.}
\end{Rmk}
We are now ready for the proof of Theorem \ref{thm:card cap hel}.
\begin{proof}[Proof of Theorem \ref{thm:card cap hel}]\setcounter{claim}{0} Assume that there exist $A_0,(M_a,\om_a)_{a\in A_0}$ as in the hypothesis of \reff{thm:card cap hel:caps}. Let $a\in A_0\cap(0,\infty)$. We define
\[(W_a,\eta_a):=\left(M_a\disj M_{-a},\om_a\disj\om_{-a}\right).\]
Let $A\in\PP\big(A_0\cap(0,\infty)\big)$. Recall the definition \eqref{eq:OO0} of $\OO_0$. We define the function
\[c_A:=\sup_{a\in A}c_{W_a,\eta_a}:\OO_0\to[0,\infty].\]
If $k=2$ and the ball $B$ lies in $\OO$, then we define the function $\cc_A:$ by
\begin{equation}\label{eq:c A M om}\cc_A:=\max\left\{c_A,w\right\}:\OO_0\to[0,\infty].\end{equation}
The functions $c_A$ and $\cc_A$ are generalized capacities on $\CCC$.
\begin{claim}\label{claim:c A} 
\begin{enui}\item\label{claim:c A:A A'} The map $\PP\big(A_0\cap(0,\infty)\big)\ni A\mapsto c_A\in\caps(\CCC)$ is injective.\\

\noindent Assume now that the hypotheses of Theorem \ref{thm:card cap hel}\reff{thm:card cap hel:ncaps} are satisfied.
\item\label{claim:c A:cc A} The map $\PP\big(A_0\cap(0,\infty)\big)\ni A\mapsto\cc_A\in\caps(\CCC)$ is injective.
\item\label{claim:c A:norm} For every $A\in\PP\big(A_0\cap(0,\infty)\big)$ the capacity $\cc_A$ is normalized. 
\end{enui}
\end{claim}
\begin{proof}[Proof of Claim \ref{claim:c A}] We denote
\[h_M:=h_{M,\oo,\om},\quad f_a:=h_{M_a},\quad f:=(f_a)_{a\in A_0},\]
and define $C^f_0,C^f_1$ as in (\ref{eq:C f0},\ref{eq:C f1}). Let $a\neq a'\in A_0\cap(0,\infty)$, and $c\in(0,\infty)$, such that there exists a $\CCC$-morphism $\phi$ from $(W_a,c\eta_a)$ to $(W_{a'},\eta_{a'})$.

\textbf{Case A:} There exist such a $\phi$ and $b\in\{a,-a\}$, $b'\in\{a',-a'\}$, such that $b>b'$ and $\phi(M_b)\sub M_{b'}$. We denote
\[M:=M_b,\quad\om:=\om_b,\quad M':=M_{b'},\quad\om':=\om_{b'},\quad I:=I_M,\quad I':=I_{M'}.\]
Let $d\in A_0$. By hypotheses $M_d$ is nonempty, compact, and 1-connected. Since by hypothesis $n\geq2>0$ and $\om_d$ is maxipotent and exact, we have $\bd M_d\neq\emptyset$. Hence the hypotheses of Lemma \ref{le:part}\reff{le:part:conn} are satisfied. Applying this lemma, it follows that $\p^\phi$ is an $(I,I')$-partition. By Lemma \ref{le:hel} the set $\p^\phi$ is an $\big(h_M,h_{M'},c^n\big)$-partition. It follows that
\begin{equation}\label{eq:c n C f0}c^n\leq C^f_0.
\end{equation}

Consider now the \textbf{case} that is complementary to Case A. Then $a<a'$ and there exists a morphism $\phi$ from $(W_a,c\eta_a)$ to $(W_{a'},\eta_{a'})$, such that $\phi(W_a)\sub M_{a'}$. Lemmata \ref{le:part}\reff{le:part:two} and \ref{le:hel} imply that $\p^\phi$ is an $\big(h_{M_a},h_{M_{-a}},h_{M_{a'}},c^n\big)$-partition. It follows that $c^n\leq C^f_1$. Combining this with \reff{eq:c n C f0}, in any case we have
\[c^n\leq C:=\max\left\{C^f_0,C^f_1\right\}.\]
It follows that
\begin{align*}&\sup\big\{c\in(0,\infty)\,\big|\,\exists a\neq a'\in A_0\cap(0,\infty)\,\exists\textrm{ morphism }(W_a,c\eta_a)\to(W_{a'},\eta_{a'})\big\}\\
\leq&\sqrt[n]{C}\\
<&1\qquad\textrm{(using our hypothesis \reff{thm:card cap hel:helicity} and Definition \ref{defi:I collection}).}
\end{align*}
It follows that
\begin{equation}\label{eq:c A}c_A(W_{a'},\eta_{a'})<1,\quad\forall A\in\PP\big(A_0\cap(0,\infty)\big),\quad a'\in A_0\cap(0,\infty)\wo A.%
\footnote{A priori the function $c:=c_A$ is only defined on the set $\OO_0$. For a general $(M,\om)\in\OO$ we define $c(M,\om):=c(M_0,\om_0)$, where $(M_0,\om_0)$ is an arbitrary object of $\OO_0$ isomorphic to $(M,\om)$.}%
\end{equation}
Let $A\neq A'\in\PP\big(A_0\cap(0,\infty)\big)$. Assume first that $A'\wo A\neq\emptyset$. We choose $a'\in A'\wo A$. Since $c_{A'}(W_{a'},\eta_{a'})\geq1$,\footnote{In fact equality holds, but we do not use this.} inequality \reff{eq:c A} implies that $c_A\neq c_{A'}$. This also holds in the case $A\wo A'\neq\emptyset$, by an analogous argument. This proves statement \reff{claim:c A:A A'}.\\

We prove \textbf{\reff{claim:c A:cc A}}. Combining inequality \eqref{eq:c A} with our hypothesis \reff{thm:card cap hel:w}, we have
\[\cc_A(W_{a'},\eta_{a'})<1,\quad\forall A\in\PP\big(A_0\cap(0,\infty)\big),\quad a'\in A_0\cap(0,\infty)\wo A.\]
Hence an argument as above shows that the map $\PP\big(A_0\cap(0,\infty)\big)\ni A\mapsto\cc_A$ is injective. This proves \reff{claim:c A:cc A}.\\

We prove \reff{claim:c A:norm}. Let $A\in\PP\big(A_0\cap(0,\infty)\big)$. By our definition \reff{eq:c A M om} we have
\begin{equation}\label{eq:pi}\pi=w(B)\leq\cc_A(B).\end{equation}
Since $B$ symplectically embeds into $Z$, we have $c_{M,\om}(B)\leq c_{M,\om}(Z)$ for every object $(M,\om)$ of $\CCC$. It follows that
\begin{equation}\label{eq:c A B 2n}\cc_A(B)\leq\cc_A(Z).\end{equation} Our hypothesis \reff{thm:card cap hel:Z} and Gromov's Nonsqueezing Theorem imply that $\cc_A(Z)\leq\pi$. Combining this with (\ref{eq:pi},\ref{eq:c A B 2n}), it follows that $\cc_A$ is normalized. This proves \reff{claim:c A:norm} and therefore Claim \ref{claim:c A}.
\end{proof}
Claim \ref{claim:c A}\reff{claim:c A:A A'} implies that
\begin{equation}\label{eq:caps Scat beth2}\vert\caps(\CCC)\vert\geq\left\vert\PP\big(A_0\cap(0,\infty)\big)\right\vert=\beth_2,\end{equation}
where in the second inequality we used our hypothesis that $A_0$ is an interval of positive length. On the other hand, by Corollary \ref{cor:card sec} in the appendix the set $\OO_0$ has cardinality at most $\beth_1$. It follows that
\[|\caps(\CCC)|\leq\left|[0,\infty]^{\OO_0}\right|\leq\beth_1^{\beth_1}=\beth_2.\]
Combining this with \reff{eq:caps Scat beth2}, the statement of Theorem \ref{thm:card cap hel}\reff{thm:card cap hel:caps} follows.

The statement of Theorem \ref{thm:card cap hel}\reff{thm:card cap hel:ncaps} follows from an analogous argument, using parts (\ref{claim:c A:cc A},\ref{claim:c A:norm}) of Claim \ref{claim:c A}. This completes the proof of Theorem \ref{thm:card cap hel}.
\end{proof}
\arxiv{\section{Proof of Proposition \ref{prop:I collection} (sufficient conditions for being an $I$-collection)} \label{sec:proof:prop:I collection}
\begin{proof}[Proof of Proposition \ref{prop:I collection}]\setcounter{claim}{0} Let $I=(I_a),f=(f_a)$ be as in the hypothesis. To simplify notation, we canonically identify the collection $f$ with its disjoint union $\Disj f:\Disj I\to\R$.
\begin{claim}\label{claim:a a'} Let $a,a'\in A_0$. If $a>a'$ then for every partition $\p$ of $I_a\disj I_{a'}$ there exists $J\in\p$, such that
\begin{equation}\label{eq:sum i J Ia}\sum_{i\in J\cap I_a}f(i)>\sum_{i'\in J\cap I_{a'}}f(i').
\end{equation}
\end{claim}
\begin{proof}[Proof of Claim \ref{claim:a a'}] This follows from hypothesis \reff{eq:sum a sum a'}.
\end{proof}
By hypothesis \reff{eq:Ia} there exists $k$, such that $|I_a|=k+1$, for every $a\in A_0$. By hypothesis \reff{eq:f-1 0 infty} for every $a\in A_0$ the set $f_a^{-1}((0,\infty))$ contains a unique element $p_a$. Hypotheses (\ref{eq:sum a leq1},\ref{eq:f -1}) imply that
\begin{equation}\label{eq:h Pa leq}f(p_a)\leq k+1,\quad\forall a\in A_0.
\end{equation}
Recalling the notation \reff{eq:sum f}, we have
\begin{align}\label{eq:inf>0}\inf_{a\in A_0}\sum f_a&>0,\qquad\textrm{(using (\ref{eq:sup},\ref{eq:-1}))}\\
\label{eq:f pa >1}f(p_a)&>1,\quad\forall a\in A_0\qquad\textrm{(using (\ref{eq:inf>0},\ref{eq:-1})).}
\end{align}
\begin{claim}\label{claim:sup<2inf+1} If $k=1$ or $2$ then the inequality \reff{eq:sup<2inf+1} holds.
\end{claim}
\begin{proof} For every $a\in A_0$ we have
\begin{align*}f(p_a)&=\sum f_a-\sum_{n\in I_a\wo\{p_a\}}f(n)\\
&\geq\inf_b\sum f_b+1-(k-1)\sup\big(\im(f)\cap(-\infty,0]\big)\qquad\textrm{(using \reff{eq:-1})}\\
&>k+(2-k)\inf_b\sum f_b\qquad\textrm{(using \reff{eq:sup})}\\
&\geq k\qquad\textrm{(using that $k=1$ or $2$, and \reff{eq:inf>0}).}
\end{align*}
Using \reff{eq:h Pa leq}, it follows that \reff{eq:sup<2inf+1} holds. This proves Claim \ref{claim:sup<2inf+1}.
\end{proof}

We now check the conditions (\ref{defi:I collection:sup >},\ref{defi:I collection:sup <}) of Definition \ref{defi:I collection}.\\

\textbf{Condition \reff{defi:I collection:sup >}:} Let $a,a'\in A_0$ be such that $a>a'$, $C\in(0,\infty)$ and $\p$ be an $\big(f_a,f_{a'},C\big)$-partition. If $C\geq1$ then Claim \ref{claim:a a'} implies that condition \reff{eq:C sum f i} in Definition \ref{defi:part} with $I:=I_a$, $I':=I_{a'}$ is violated. It follows that $C<1$. 

We denote by $J_0$ the unique element of $\p$ containing $p_a$.
\begin{claim}\label{claim:p a'} We have $p_{a'}\in J_0$. 
\end{claim}
\begin{proof}[Proof of Claim \ref{claim:p a'}] By Definition \ref{defi:part} we have $\vert J_0\cap I_a\vert=1$. It follows that $J_0\cap I_a=\{p_a\}$. Therefore, by condition \reff{eq:C sum f i} applied to $J:=J_0$, we have  
\[Cf(p_a)\leq\sum_{i'\in J_0\cap I_{a'}}f(i').\]
Since $Cf(p_a)>0$ and $p_{a'}$ is the only point in $I_{a'}$ at which $f$ is positive, Claim \ref{claim:p a'} follows.
\end{proof}
\begin{claim}\label{claim:I a'} We have $f_{a'}^{-1}(-1)\sub J_0$.
\end{claim}
\begin{proof}[Proof of Claim \ref{claim:I a'}] Let $J\in\p\wo\{J_0\}$. By \reff{eq:J cap I} the set $J\cap I_a$ consists of a unique element $i$. Hypothesis \reff{eq:f -1} and the inequality $C<1$ imply that $Cf(i)>-1$. Combining this with \reff{eq:C sum f i}, it follows that
\begin{equation}\label{eq:-1 sum}\sum_{i'\in J\cap I_{a'}}f(i')>-1.\end{equation}
Since $J$ and $J_0$ are disjoint, Claim \ref{claim:p a'} implies that $p_{a'}\not\in J$. Therefore, \reff{eq:-1 sum} implies that $J\cap I_{a'}\cap f^{-1}(-1)=\emptyset$. Since this holds for every $J\in\p\wo\{J_0\}$, and $\p$ covers $I_{a'}$, it follows that $I_{a'}\cap f^{-1}(-1)\sub J_0$. This proves Claim \ref{claim:I a'}.
\end{proof}
Claims \ref{claim:p a'},\ref{claim:I a'} and hypothesis \reff{eq:-1} imply that $\vert J_0\cap I_{a'}\vert\geq2$. Since $|I_a|=|I_{a'}|=k+1$ and $p_a\in J_0\cap I_a$, it follows that
\begin{equation}\label{eq:Ia Ia'}\big|\big(I_a\disj I_{a'}\big)\wo J_0\big|\leq2k-1.\end{equation}
The condition \reff{eq:J cap I} implies that $\big\vert\p\wo\{J_0\}\big\vert=\vert I_a\vert-1=k$. Since the elements of $\p\wo\{J_0\}$ are disjoint and their union is contained in $\big(I_a\disj I_{a'}\big)\wo J_0$, using \reff{eq:Ia Ia'}, it follows that there exists $J_1\in\p\wo\{J_0\}$ satisfying $\vert J_1\vert\leq1$. Since $\vert J_1\cap I_a\vert=1$, it follows that
\begin{equation}\label{eq:J I a'}J_1\cap I_{a'}=\emptyset.
\end{equation}
The facts $J_1\neq J_0$, and that $p_a$ lies in $J_0$ and is the only point of $I_a$ at which $f$ is positive, imply that $\sum_{i\in J_1\cap I_a}f(i)\leq\sup\big(\im(f)\cap(-\infty,0]\big)$. Using \reff{eq:J I a'} and recalling the definition \reff{eq:Sum J f f' C}, it follows that
\begin{equation}\label{eq:Sum J1}\Sum{J_1,f_a,f_{a'},C}\geq-C\sup\big(\im(f)\cap(-\infty,0]\big).\end{equation}
Summing up the inequality \reff{eq:C sum f i} over all $J\in\p\wo\{J_1\}$ and adding \reff{eq:Sum J1}, we obtain
\[-C\sum f_a+\sum f_{a'}\geq-C\sup\big(\im(f)\cap(-\infty,0]\big).\]
It follows that
\begin{align*}C\left(-\sup\big(\im(f)\cap(-\infty,0]\big)+\inf_a\sum f_a\right)&\leq\sum f_{a'}\\
&\leq1\qquad\textrm{(using hypothesis \reff{eq:sum a leq1}).}
\end{align*}
Combining this with hypothesis \reff{eq:sup}, it follows that $C^f_0<1$. Hence $f$ satisfies \reff{defi:I collection:sup >}.\\   

\textbf{Condition \reff{defi:I collection:sup <}:} Let $a,a'\in(0,\infty)$, such that $a<a'$, $C\in(0,\infty)$ and $\p$ be an $\big(f_a,f_{-a},f_{a'},C\big)$-partition. We denote by $J_0\in\p$ the unique element that contains $p_a$. We will show that $\p$ and $J_0$ look like in Figure \ref{fig:partition}.
\begin{figure}[H]
\centering
\leavevmode\scalebox{.75}{\epsfbox{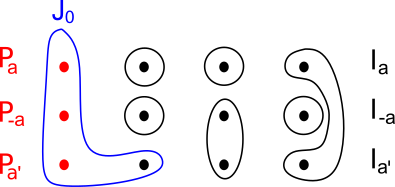}}
\caption{The dots in the first row constitute the set $I_a$, which contains the point $p_a$, and similarly for $I_{-a}$ and $I_{a'}$. The blue and black sets denote the elements of the partition $\p$. We show below that except for $p_a$, the blue set $J_0$ also contains $p_{-a},p_{a'}$, and an element of $I_{a'}$ at which $f$ takes on the value $-1$. Note that $J_0$ intersects both $I_a$ and $I_{-a}$ in exactly one point, and that the other elements of $\p$ intersect $I_a\disj I_{-a}$ in exactly one point.}
\label{fig:partition}
\end{figure}
\begin{claim}\label{claim:J0} We have $p_{a'},p_{-a}\in J_0$.
\end{claim}
\begin{proof}[Proof of Claim \ref{claim:J0}] We show that $p_{a'}\in J_0$. Conditions (\ref{defi:I+ I- I':J},\ref{defi:I+ I- I':other}) of Definition \ref{defi:I+ I- I'} with $I^\pm:=I_{\pm a}$ imply that $J_0\cap I_{\pm a}$ is empty or a singleton. Combining this with the fact that $p_a\in J_0$, hypothesis \reff{eq:f -1}, and \reff{eq:f pa >1}, we obtain
\[\sum_{i\in J_0\cap(I_a\disj I_{-a})}f(i)>0.\]
Using condition \reff{eq:C sum f i} with $J=J_0$, it follows that $p_{a'}\in J_0$.

To show that $p_{-a}\in J_0$, let $J\in\p\wo\{J_0\}$. Since $p_{a'}\in J_0$, it does not lie in $J$. It follows that $\sum_{i'\in J\cap I_{a'}}f(i')\leq0$. Using \reff{eq:C sum f i} with $I=I_a\disj I_{-a}$, it follows that
\begin{equation}\label{eq:sum i J Ia I-a}\sum_{i\in J\cap(I_a\disj I_{-a})}f(i)\leq0.\end{equation}
Conditions (\ref{defi:I+ I- I':J},\ref{defi:I+ I- I':other}) of Definition \ref{defi:I+ I- I'} with $I^\pm:=I_{\pm a}$ imply that $J\cap I_{\pm a}$ is empty or a singleton. Using hypothesis \reff{eq:f -1} and \reff{eq:sum i J Ia I-a}, it follows that $J\cap I_{-a}$ is empty or consists of one element $i$, satisfying $f(i)\leq1$. Using \reff{eq:f pa >1}, it follows that $p_{-a}\not\in J$. Since this holds for every $J\in\p\wo\{J_0\}$, it follows that $p_{-a}\in J_0$. This proves Claim \ref{claim:J0}.
\end{proof}
\begin{claim}\label{claim:C<1} We have $C<1$.
\end{claim}
\begin{proof}[Proof of Claim \ref{claim:C<1}] By Remark \ref{rmk:I+ I- I':card} we have $|\p|=2k+1$. Since $|I_{a'}|=k+1$, $k\geq1$, and the elements of $\p$ are disjoint, it follows that there exists $J_1\in\p$, such that
\begin{equation}\label{eq:J1 I a'}J_1\cap I_{a'}=\emptyset.
\end{equation}
Claim \ref{claim:J0} implies that $J_1\neq J_0$, and hence that $p_a,p_{-a}\not\in J_1$. By Definition \ref{defi:I+ I- I'}\reff{defi:I+ I- I':other} we have
\begin{equation}\label{eq:n}J_1\cap I_a\disj I_{-a}=\{n\},\quad\textrm{for some point }n.
\end{equation}
By \reff{eq:sup} we have 
\begin{equation}\label{eq:f n}f(n)<-1+\inf_{b\in A_0}\sum f_b.
\end{equation}
Denoting
\[\Sum{J}:=\sum_{i\in J\cap(I_a\disj I_{-a})}f(i),\quad\sump_J:=\sum_{i'\in J\cap I_{a'}}f(i'),\]
we have
\begin{align*}1&\geq\sum f_{a'}\qquad\textrm{(using \ref{eq:sum a leq1})}\\
&=\sum_{J\in\p}\sump_J\\
&=\sum_{J\in\p}\left(-C\Sum{J}+\sump_J\right)+C\sum\big(f_a+f_{-a}\big)\\
&\geq-C\Sum{J_1}+\sump_{J_1}+2C\inf_{b\in A_0}\sum f_b\qquad\textrm{(using \reff{eq:C sum f i} with $J\in\p\wo\{J_1\}$)}\\
&>C\left(1+\inf_{b\in A_0}\sum f_b\right)\qquad\textrm{(using (\ref{eq:n},\ref{eq:f n},\ref{eq:J1 I a'}))}.
\end{align*}
Using \reff{eq:inf>0}, it follows that $C<1$. This proves Claim \ref{claim:C<1}.
\end{proof}
\begin{claim}\label{claim:I a' } We have $f_{a'}^{-1}(-1)\sub J_0$.
\end{claim}
\begin{proof}[Proof of Claim \ref{claim:I a' }] Let $J\in\p\wo\{J_0\}$. By Claim \ref{claim:J0} we have $p_{-a}\in J_0$. Since also $p_a\in J_0$, by Definition \ref{defi:I+ I- I'}\reff{defi:I+ I- I':other}, it follows that $\big|J\cap(I_a\disj I_{-a})\big|=1$. Using hypothesis \reff{eq:f -1} and \reff{eq:C sum f i}, it follows that
\begin{align}\nn\sum_{i'\in J\cap I_{a'}}f(i')&\geq-C\\
\label{eq:sum i' J}&>-1\qquad\textrm{(by Claim \ref{claim:C<1}).}
\end{align}
By Claim \ref{claim:J0} we have $p_{a'}\in J_0$. Hence this point does not lie in $J$. Therefore, \reff{eq:sum i' J} implies that $J\cap I_{a'}\cap f^{-1}(-1)=\emptyset$. Since this holds for every $J\in\p\wo\{J_0\}$, and $\p$ covers $I_{a'}$, it follows that $I_{a'}\cap f^{-1}(-1)\sub J_0$. This proves Claim \ref{claim:I a' }.
\end{proof}
Claim \ref{claim:J0} and Definition \ref{defi:I+ I- I'}\reff{defi:I+ I- I':J} imply that $J_0\cap\big(I_a\disj I_{-a}\big)=\{p_a,p_{-a}\}$, and therefore,
\begin{equation}\label{eq:sum i J0}\sum_{i\in J_0\cap\big(I_a\disj I_{-a}\big)}f(i)=f(p_a)+f(p_{-a}).
\end{equation}
Claim \ref{claim:I a' } and hypothesis \reff{eq:-1} imply that
\[\sum_{i'\in J_0\cap I_{a'}}f(i')\leq f(p_{a'})-1.\]
Combining this with \reff{eq:sum i J0} and \reff{eq:C sum f i} with $J=J_0$, it follows that
\[C(f(p_a)+f(p_{-a}))\leq f(p_{a'})-1.\]
It follows that 
\begin{align*}C&\leq\frac{f(p_{a'})-1}{f(p_a)+f(p_{-a})}\\
&\leq\frac{\sup_bf(p_b)-1}{2\inf_bf(p_b)}\\
&<1\qquad\textrm{(using \reff{eq:sup<2inf+1}).}
\end{align*}
Here in the case $k=1$ or $2$ we use Claim \ref{claim:sup<2inf+1}. It follows that $C^f_1<1$. Hence $f$ satisfies \reff{defi:I collection:sup <}. This completes the proof of Proposition \ref{prop:I collection}.\end{proof}
}
\section{Proof of Theorem \ref{thm:card cap gen}\reff{thm:card cap gen:gen} 
(cardinality of a generating set)}\label{sec:proof:thm:card cap gen:gen}
The proof of Theorem \ref{thm:card cap gen}\reff{thm:card cap gen:gen} is based on the following lemma. For every set $S$ we denote by $\PP(S)$ its power set. For every subcollection $\CC\sub\PP(X)$ we denote by $\si(\CC)$ the $\si$-algebra generated by $\CC$. It is given by 
\[\si(\CC):=\bigcap_{\A\,\si\textrm{-algebra on }X:\,\CC\sub\A}\A.\]
A measurable space is a pair $(X,\A)$, where $X$ is a set and $\A$ a $\si$-algebra on $X$. Let $(X,\A),(X',\A')$ be measurable spaces. A map $f:X\to X'$ is called \emph{$(\A,\A')$-measurable} iff $f^{-1}(A')\in\A$, for all $A'\in\A'$. We denote
\[\M(\A,\A'):=\big\{(\A,\A')\textrm{-measurable map: }X\to X'\big\}.\]
\begin{lemma}[cardinality of the set of measurable maps]\label{le:card maps} Let $X,X'$ be sets and $\CC\sub\PP(X),\CC'\sub\PP(X')$ be subcollections. Assume that $|\CC|\leq\beth_1$, $|\CC'|\leq\beth_0=\aleph_0$, and
\begin{equation}\label{eq:bigcap}\forall x'\in X':\quad\bigcap_{C'\in\CC':\,x'\in C'}C'=\{x'\}.\end{equation}
We define $\A:=\si(\CC)$, $\A':=\si(\CC')$. Then $\M(\A,\A')$ has cardinality at most $\beth_1$.
\end{lemma}
For the proof of this lemma we need the following.
\begin{lemma}[cardinality of $\si$-algebra]\label{le:card si alg} Let $X$ be a set and $\CC\sub\PP(X)$ be a subcollection of cardinality at most $\beth_1$. Then $\si(\CC)$ has cardinality at most $\beth_1$.
\end{lemma}
The proof of this lemma is based on the following. Let $S$ be a set, $F:\PP(S)\to\PP(S)$, such that
\begin{equation}\label{eq:F A}A\sub F(A),\quad\forall A\in\PP(S).\end{equation}
Let $A\in\PP(S)$. We define $\lan F,A\ran$, the \emph{set generated by $F,A$}, to be the smallest fixed point of $F$ containing $A$. This is the set given by
\[\lan F,A\ran=\bigcap\big\{B\in\PP(S)\,\big|\,A\sub B=F(B)\big\}.\footnote{This intersection is well-defined, since the collection of all admissible $B$ is nonempty. It contains $B=S$.}\]
\begin{lemma}[cardinality of generated set]\label{le:card gen} The set $\lan F,A\ran$ has cardinality at most $\beth_1$, if the following conditions are satisfied:
\begin{enua}
\item\label{le:card gen:mon} $F$ is monotone, i.e., $B\sub C$ implies that $F(B)\sub F(C)$.
\item\label{le:card gen:beth1} $|A|\leq\beth_1$.
\item\label{le:card gen:F beth1} If $|B|\leq\beth_1$ then $|F(B)|\leq\beth_1$, for every $B\in\PP(S)$.
\item\label{le:card gen:fixed} If $B\in\PP(S)$ satisfies
\begin{equation}\label{eq:F bigcup B}F(C)\sub B,\quad\forall\textrm{ countable subset }C\sub B,
\end{equation}
then $B$ is a fixed point of $F$.
\end{enua}
\end{lemma}
\begin{proof}[Proof of Lemma \ref{le:card gen}]\setcounter{claim}{0} We denote by $\om_1$ the smallest uncountable (von Neumann) ordinal, i.e., the set of countable ordinals. We define $A_0:=A$, and using transfinite recursion, for every $\al\leq\om_1$, we define
\begin{equation}\label{eq:A al}A_\al:=\left\{\begin{array}{ll}
F(A_\be),&\textrm{if }\al=\be+1,\\
\bigcup_{\be<\al}A_\be,&\textrm{if $\al\neq0$ is a limit ordinal.}
\end{array}\right.
\end{equation}
(A limit ordinal is an ordinal for which there does not exist any ordinal $\be$ for which $\al=\be+1$.)
\begin{claim}\label{claim:B} We have
\[\lan F,A\ran\sub A_{\om_1}.\]
\end{claim}
\begin{proof}[Proof of Claim \ref{claim:B}] Since $A_0\sub A_{\om_1}$, it suffices to show that $A_{\om_1}$ is a fixed point of $F$.
\begin{claim}\label{claim:F bigcup B} Condition \eqref{eq:F bigcup B} is satisfied with $B=A_{\om_1}$.
\end{claim}
\begin{proof}[Proof of Claim \ref{claim:F bigcup B}] Let $C\sub A_{\om_1}$ be a countable subset. The definition \reff{eq:A al}, condition \eqref{eq:F A}, and transfinite induction imply that for every pair $\al,\be$ of ordinals, we have
\begin{equation}\label{eq:al be}\al\leq\be\then A_\al\sub A_\be.\end{equation}
We choose a collection $(\al_c)_{c\in C}$ of countable ordinals, such that $c\in A_{\al_c}$, for every $c\in C$. The ordinal
\[\al:=\sup_{c\in C}\al_c:=\bigcup_{c\in C}\al_c\]
is countable, and therefore less than $\om_1$. For every $c\in C$, we have $\al_c\leq\al$, and thus by \eqref{eq:al be}, $A_{\al_c}\sub A_\al$. It follows that $C\sub A_\al$, and therefore,
\begin{align*}F(C)&\sub F(A_\al)\qquad\textrm{(using \reff{le:card gen:mon})}\\
&=A_{\al+1}\qquad\textrm{(using \eqref{eq:A al})}\\
&\sub A_{\om_1}\qquad\textrm{(using $\al+1<\om_1$ and \eqref{eq:al be})}.
\end{align*}
This proves Claim \ref{claim:F bigcup B}.
\end{proof}
By this claim and \reff{le:card gen:fixed} the set $A_{\om_1}$ is a fixed point of $F$. This proves Claim \ref{claim:B}.
\end{proof}
For every ordinal $\al$ we denote by $P(\al)$ the statement ``$|A_\al|\leq\beth_1$''.
\begin{claim}\label{claim:A al} The statement $P(\al)$ is true for all $\al\leq\om_1$.
\end{claim}
\begin{proof}[Proof of Claim \ref{claim:A al}] We prove this by transfinite induction. Let $\al\leq\om_1$ and assume that the statement holds for all $\be<\al$. If $\al=0$ then $P(0)$ holds by our hypothesis \reff{le:card gen:beth1}. If $\al=\be+1$ for some $\be$ then $P(\al)$ holds by \eqref{eq:A al} and our hypothesis \reff{le:card gen:F beth1}. If $\al\neq0$ is a limit ordinal, then $P(\al)$ holds by \eqref{eq:A al}, our induction hypothesis, and the fact $|\al|\leq|\om_1|\leq\beth_1$. This completes the inductive step. Claim \ref{claim:A al} now follows from transfinite induction.
\end{proof}
Lemma \ref{le:card gen} follows from Claims \ref{claim:B} and \ref{claim:A al}.
\end{proof}
\begin{proof}[Proof of Lemma \ref{le:card si alg}] This follows from Lemma \ref{le:card gen} with
\[S:=\PP(X),\,A:=\CC,\,F(\D):=\left\{\bigcup\E\,\Big|\,\E\sub\D\textrm{ countable}\right\}\cup\big\{X\wo E\,\big|\,E\in\D\big\}.\]
To see that \reff{le:card gen:fixed} holds, let $B=\D\in\PP(S)$ be such that \eqref{eq:F bigcup B} holds. It suffices to show that $\D$ is closed under countable unions and complements. Let $\E\sub\D$ be a countable subcollection. We have
\begin{align*}\bigcup\E&\in F(\E)\\
&\sub\D\qquad\textrm{(using \eqref{eq:F bigcup B}).}
\end{align*}
Hence $\D$ is closed under countable unions. Let now $E\in\D$. We have
\begin{align*}X\wo E&\in F(\{E\})\\
&\sub\D\qquad\textrm{(using \eqref{eq:F bigcup B}).}
\end{align*}
Hence $\D$ is closed under complements. It follows that $\D$ is a fixed point of $F$. This proves \reff{le:card gen:fixed} and completes the proof of Lemma \ref{le:card si alg}.
\end{proof}
\begin{proof}[Proof of Lemma \ref{le:card maps}] Recall that for every pair of sets $S,S'$ we denote by ${S'}^S$ the set of maps from $S$ to $S'$. Let $f\in\M(\A,\A')$ and $x'\in X'$. Our hypothesis that $|\CC'|\leq\aleph_0$ and \eqref{eq:bigcap} imply that the set $\{x'\}$ is a countable intersection of elements of $\CC'$. Hence it lies in $\A'$. It follows that $f^{-1}(x')\in\A$. The following map is therefore well-defined:
\[\iota:\M(\A,\A')\to\A^{X'},\quad\iota(f)(x'):=f^{-1}(x').\]
We define the map
\[\phi:\M(\A,\A')\to\A^{\CC'},\quad\phi(f)(C'):=f^{-1}(C'),\]
\[\psi:\A^{\CC'}\to\A^{X'},\quad\psi(A)(x'):=\bigcap_{C'\in\CC':\,x'\in C'}A(C').\]
Our hypothesis $|\CC'|\leq\aleph_0$ implies that $\psi(A)(x')$ is a countable intersection of elements of $\A$, hence an element of $\A$. It follows that $\psi$ is well-defined. For every $f\in\M(\A,\A')$ and $x'\in X'$, we have 
\begin{align*}\iota(f)(x')&=f^{-1}(x')\\
&=f^{-1}\left(\bigcap_{C'\in\CC':\,x'\in C'}C'\right)\qquad\textrm{(by \reff{eq:bigcap})}\\
&=\bigcap_{C'\in\CC':\,x'\in C'}f^{-1}(C')\\
&=\big(\psi(\phi(f))\big)(x').
\end{align*}
Hence the equality $\iota=\psi\circ\phi$ holds. Since $\iota$ is injective, it follows that $\phi$ is injective. Our hypothesis that $|\CC|\leq\beth_1$ and Lemma \ref{le:card si alg} imply that $\big|\A=\si(\CC)\big|\leq\beth_1$. Since $|\CC'|\leq\aleph_0$, it follows that $\big|\A^{\CC'}\big|\leq\beth_1$. Since $\phi$ maps $\M(\A,\A')$ to $\A^{\CC'}$, it follows that $\big|\M(\A,\A')\big|\leq\beth_1$. This proves Lemma \ref{le:card maps}.
\end{proof}

In the proof of Theorem \ref{thm:card cap gen}\reff{thm:card cap gen:gen} we will also use the following.
\begin{rmks}\label{rmk:top}\textnormal{
\begin{enui}
\item\label{rmk:top:prod} Every countable product of second countable topological spaces is second countable.
\item\label{rmk:top:card} Let $(X,\tau)$ be a topological space and $\B$ a basis of $\tau$. Then the following inequality holds:
\[|\tau|\leq2^{|\B|}\]
\end{enui}
}
\end{rmks}

\begin{proof}[Proof of Theorem \ref{thm:card cap gen}\reff{thm:card cap gen:gen}]\setcounter{claim}{0} Let $\G_0$ be a countable subset of $X^S$. We equip $X^{\G_0}$ with the product topology $\tau_{\G_0}$. We define $\A_{\G_0},\A$ to be the Borel $\si$-algebras of $\tau_{\G_0},\tau$.
\begin{claim}\label{claim:M} The set $\M(\A_{\G_0},\A)$ has cardinality at most $\beth_1$.
\end{claim}
\begin{proof}[Proof of Claim \ref{claim:M}] Our assumption that $\tau$ is separable and metrizable, implies that it is second countable. Hence by Remark \ref{rmk:top}\reff{rmk:top:prod}, the same holds for $\tau_{\G_0}$. Hence by Remark \ref{rmk:top}\reff{rmk:top:card}, we have
\begin{equation}\label{eq:tau gs0}|\tau_{\G_0}|\leq2^{\aleph_0}=\beth_1.\end{equation}
We have $\A_{\G_0}=\si(\tau_{\G_0})$. Since $\tau$ is separable, there exists a countable $\tau$-dense subset $A$ of $X$. We define $\CC$ to be the collection of all open balls with rational radius around points in $A$. Since $A$ is dense, every element of $\tau$ is a union of elements of $\CC$. Since $A$ is countable, the set $\CC$ is countable. It follows that $\A=\si(\tau)=\si(\CC)$. Since $\tau$ is separable and metrizable, the condition \reff{eq:bigcap} with $\CC'$ replaced by $\CC$ is satisfied. Using \reff{eq:tau gs0}, it follows that the hypotheses of Lemma \ref{le:card maps} are satisfied with $\CC,\CC'$ replaced by $\tau_{\G_0},\CC$. Applying this lemma, it follows that $\big|\M(\A_{\G_0},\A)\big|\leq\beth_1$. This proves Claim \ref{claim:M}.
\end{proof}
Let $\G$ be a subset of $X^S$ of cardinality at most $\beth_1$. By Definition \ref{defi:gen set} the set countably Borel-generated by $\G$ is given by 
\[\lan\G\ran:=\big\{f\circ ev_{\G_0}\,\vert\,\G_0\sub\G\textrm{ countable, }f\in\M(\A_{\G_0},\A)\big\}.\]
The set of all countable subsets of $\G$ has cardinality at most $\beth_1^{\aleph_0}=\beth_1$. Using Claim \ref{claim:M}, it follows that
\[|\lan\G\ran|\leq\beth_1^2=\beth_1.\]
This proves Theorem \ref{thm:card cap gen}\reff{thm:card cap gen:gen}.
\end{proof}
\section{Proof of Theorem \ref{thm:gen uncount} (uncountability of every generating set under a mild hypothesis)} \label{proof:thm:gen uncount}
\begin{proof}[Proof of Theorem \ref{thm:gen uncount}]\setcounter{claim}{0} Let $\CCC=(\OO,\M),A,M$ be as in the hypothesis. \WLOG we may assume that $A$ is open. Our hypothesis \reff{eq:Vol M} implies that the function $\Vol^{\frac1n}\circ M:A\to\R$ is continuous and strictly increasing. Hence it is injective with image $\wt A$ given by an interval. We define
\[\wt M:=M\circ\left(\Vol^{\frac1n}\circ M\right)^{-1}:\wt A\to\OO.\]
Let $\wt a_0\in\wt A$. We define
\[g_{\wt a_0}:=c_{\wt M_{\wt a_0}}\circ\wt M:\wt A\to\R.\]
\begin{claim}\label{claim:g wt a0} This function is not differentiable at $\wt a_0$.
\end{claim}
\begin{proof}[Proof of Claim \ref{claim:g wt a0}] We have
\[\Vol^{\frac1n}\circ\wt M=\id.\]
It follows that
\begin{equation}\label{eq:g wt a0}g_{\wt a_0}(\wt a)\leq\frac{\wt a}{\wt a_0},\quad\forall\wt a\in\wt A\cap(0,\wt a_0).
\end{equation}
Our hypothesis \reff{eq:c Ma Ma'} implies that
\[g_{\wt a_0}(\wt a)=1,\quad\forall\wt a\in\wt A\cap[\wt a_0,\infty).\]
Combining this with \reff{eq:g wt a0}, it follows that $g_{\wt a_0}$ is not differentiable at $\wt a_0$. This proves Claim \ref{claim:g wt a0}.
\end{proof}
Let now $\G$ be a countable subset of $\caps(\CCC)$. Let $c\in\G$. The inequality ``$\geq$'' in our hypothesis \reff{eq:c Ma Ma'} implies that the function $c\circ M$ is increasing. It follows that the same holds for $c\circ\wt M$. Therefore, by Lebesgue's Monotone Differentiation Theorem the function $c\circ\wt M$ is differentiable\footnote{in the usual sense} almost everywhere, see e.g.~\cite[p.~156, Theorem 1.6.25]{Tao}. Since $\G$ is countable, it follows that the set of all points in $\wt A$ at which the function $c\circ\wt M$ is differentiable, for every $c\in\G$, has full Lebesgue measure. Since $A$ has positive length, the same holds for $\wt A$. It follows that there exists a point $\wt a_0\in\wt A$ at which $c\circ\wt M$ is differentiable, for every $c\in\G$. 

Let $\G_0$ be a finite subset of $\G$, and $f:[0,\infty]^{\G_0}\to[0,\infty]$ a differentiable function. We define $\ev_{\G_0}$ as in \reff{eq:ev u}. Since $c\circ\wt M$ is differentiable at $\wt a_0$ for every $c\in\G_0$, the same holds for the map $\ev_{\G_0}\circ\wt M:\wt A\to[0,\infty]^{\G_0}$. It follows that the composition $f\circ\ev_{\G_0}\circ\wt M$ is differentiable at $\wt a_0$. Using Claim \ref{claim:g wt a0}, it follows that
\[f\circ\ev_{\G_0}\circ\wt M\neq g_{\wt a_0}=c_{\wt M_{\wt a_0}}\circ\wt M,\]
and therefore that $f\circ\ev_{\G_0}\neq c_{\wt M_{\wt a_0}}$. Hence $\G_0$ does not finitely differentiably generate $c_{\wt M_{\wt a_0}}$. This proves Theorem \ref{thm:gen uncount}.
\end{proof}
\appendix

\section{Cardinality of the set of equivalence classes of pairs of manifolds and forms}\label{sec:form}
In this section we prove that the set of diffeomorphism types of smooth manifolds has cardinality at most $\beth_1$. We also prove that the same holds for the set of all equivalence classes of pairs $(M,\om)$, where $M$ is a manifold, and $\om$ is a differential form on $M$. We used this in the proof of Theorem \ref{thm:card cap hel}, to estimate the cardinality of the set of (normalized) capacities from above.

In order to deal with a certain set-theoretic issue, we explain how to make the class of all diffeomorphism types a set. Let $A,B$ be sets and $S:A\to B$ a map. Let $a\in A$. We denote $S_a:=S(a)$. Recall that in ZFC ``everything'' is a set, in particular $S_a$. Recall also that the disjoint union of $S$ is defined to be
\[\Disj S:=\big\{(a,s)\,\big|\,s\in S_a\big\}.\]
We denote
\[H^n:=\big\{x\in\R^n\,\big|\,x_n\geq0\big\}.\]
Let $S$ be a set. By an atlas on $S$ we mean a subset
\[\A\sub\Disj_{U\in\PP(S)}(H^n)^U,\]
such that 
\[\bigcup_{(U,\phi)\in\A}U=S,\]
for every $(U,\phi)\in\A$ the map $\phi$ is injective, and for all $(U,\phi),(U',\phi')\in\A$ the set $\phi(U\cap U')$ is open (in $H^n$) and the transition map
\[\phi'\circ\phi^{-1}:\phi(U\cap U')\to H^n\]
is smooth. We call an atlas maximal iff it is not contained in any strictly larger atlas. By a \emph{(smooth finite-dimensional real) manifold (with boundary)} we mean a pair $M=(S,\A)$, where $S$ is a set and $\A$ is a maximal atlas on $S$, such that the induced topology is Hausdorff and second countable. We denote by $\beth_1$ the (von Neumann) cardinal $2^{\beth_0=\aleph_0}$, and by $\sim$ the diffeomorphism relation on
\begin{equation}\label{eq:M0}\M_0:=\big\{(S,\A)\,\big|\,S\sub\beth_1,\,(S,\A)\textrm{ is a manifold}\big\}.\end{equation}
This means that $M\sim M'$ iff $M$ and $M'$ are diffeomorphic. We define the \emph{set of diffeomorphism types (of manifolds)} to be
\[\MM:=\big\{\sim\textrm{-equivalence class}\big\}.\]
\begin{rmks}[diffeomorphism types]\label{rmk:diffeo types}\textnormal{
\begin{itemize}
\item The above definition overcomes the set-theoretic issue that the ``set'' of diffeomorphism classes of all manifolds (without any restriction on the underlying set) is not a set (in ZFC).
\item\label{rmk:diffeo types:bij} Every manifold $M$ is diffeomorphic to one whose underlying set is a subset of $\beth_1$. To see this, note that using second countability and the axiom of choice, the set underlying $M$ has cardinality $\leq\beth_1$. This means that there exists an injective map $f:M\to\beth_1$. Pushing forward the manifold structure via $f$, we obtain a manifold whose underlying set is a subset of $\beth_1$, as claimed.
\item By the last remark, heuristically, there is a canonical bijection between $\MM$ and the ``set'' of diffeomorphism classes of all manifolds.
\item One may understand $\MM$ in a more general way as follows. Let $\M$ be a set consisting of manifolds, such that every manifold is diffeomorphic to some element of $\M$. For example, let $S$ be a set of cardinality at least $\beth_1$ and define $\M$ to be the set of all manifolds whose underlying set is a subset of $S$. The set $\MM$ is in bijection with the set of all diffeomorphism classes of elements of $\M$.
\end{itemize}
}
\end{rmks}
\begin{prop}\label{prop:card MM} The set $\MM$ has cardinality at most $\beth_1$. 
\end{prop}
\journal{\begin{proof} This follows from an elementary argument using Whitney's embedding theorem. See the arXiv-version of this article.
\end{proof}
}
\arxiv{In the proof of this result we will use the following.
\begin{rmk}[Whitney's Embedding Theorem]\label{rmk:Whitney}\textnormal{Let $n\in\N_0$ and $M$ be a (smooth) manifold of dimension $n$. There exists a (smooth) embedding of $M$ into $\R^{2n+1}$ with closed image. To see this, consider the double $\wt M$ of $M$, which is obtained by gluing two copies of $M$ along the boundary. By Whitney's Embedding Theorem there exists an embedding of $\wt M$ into $\R^{2n+1}$ with closed image, see e.g.~\cite[2.14.~Theorem, p.~55]{Hi}%
\footnote{In this section of Hirsch's book manifolds are not allowed to have boundary. This is the reason for considering $\wt M$, rather than $M$.}
. Composing such an embedding with one of the two canonical inclusions of $M$ in $\wt M$, we obtain an embedding of $M$ into $\R^{2n+1}$ with closed image, as desired.
}
\end{rmk}
\begin{proof}[Proof of Proposition \ref{prop:card MM}]\setcounter{claim}{0} We define
\[\M:=\Disj_{m\in\N_0}\big\{\textrm{submanifold of $\R^m$}\big\}.\]
\begin{claim}\label{claim:|M|} We have $|\M|\leq\beth_1$.
\end{claim}
\begin{proof} Let $n,m\in\N_0$. The topological space $\N_0\x H^n$ is separable. Since $|\R^m|\leq\beth_1$, it follows that
\begin{equation}\label{eq:|C|}\big|C\big(\N_0\x H^n,\R^m\big)\big|\leq\beth_1.\end{equation}
Let $n\in\N_0$ and $(m,M)\in\M$, such that $M$ is of dimension $n$. Since $M$ is second countable, there exists a surjective map $\psi:\N_0\x H^n\to M$ whose restriction to $\{i\}\x H^n$ is an embedding, for every $i\in\N_0$. It follows that $M$ lies in the image of the map
\[C\big(\N_0\x H^n,\R^m\big)\to\PP(\R^m),\quad f\mapsto\im(f).\]
Combining this with \reff{eq:|C|}, it follows that $|\M|\leq\beth_1$. This proves Claim \ref{claim:|M|}.
\end{proof}
Let $n\in\N_0$. We choose an injection $\al:\R^{2n+1}\to\beth_1$, and consider the pushforward map
\[\al_*:\M\to\MM,\quad\al_*(S,\A):=\big[\al(S),\al_*\A\big].\]
Remark \ref{rmk:Whitney} implies that this map is surjective. Using Claim \ref{claim:|M|}, it follows that $|\MM|\leq\beth_1$. This proves Proposition \ref{prop:card MM}.\end{proof}
}
We define $\M_0$ as in \eqref{eq:M0},
\begin{eqnarray*}&\Om(M):=\big\{\textrm{differential form on }M\big\},&\\
&\Om_0:=\Disj_{M\in\M_0}\Om(M),&\end{eqnarray*}
the equivalence relation $\Sim$ on $\Om_0$ by
\[(M,\om)\Sim(M',\om):\iff\exists\textrm{ diffeomorphism }\phi:M\to M':\,\phi^*\om'=\om,\]
\[\textrm{and }\OM:=\Om_0/\Sim.\]
\begin{Rmk} \textnormal{Philosophically, this is the ``set'' of all equivalence classes of pairs $(M,\om)$, where $M$ is an arbitrary manifold and $\om$ is a differential form on $M$. The above definition makes this idea precise.
}
\end{Rmk}
\begin{cor}\label{cor:card sec} The set $\OM$ has cardinality at most $\beth_1$.
\end{cor}
\begin{proof}[Proof of Corollary \ref{cor:card sec}] If $M,M'$ are manifolds and $\phi:M\to M'$ is a diffeomorphism then
\begin{equation}\label{eq:phi *}\phi^*:\Om(M')\to\Om(M)\textrm{ is a bijection.}
\end{equation}
We denote by $\Pi:\Om_0\to\OM$ and $\pi:\M_0\to\MM$ the canonical projections, and by $f:\Om_0\to\M_0$, $f((M,\om)):=M$, the forgetful map. We define $F:\OM\to\MM$ to be the unique map satisfying $F\circ\Pi=\pi\circ f$. Let $\M\in\MM$. Choosing $M\in\M$, we have
\begin{align}\nn F^{-1}(\M)&=\Pi\big((F\circ\Pi)^{-1}(\M)\big)\\
\nn&=\Pi\big((\pi\circ f)^{-1}(\M)\big)\\
\nn&=\Pi\big(f^{-1}(\M)\big)\qquad\textrm{(using that $\pi^{-1}(\M)=\M$)}\\
\label{eq:Pi}&=\Pi\big(f^{-1}(M)=\Om(M)\big)\qquad\textrm{(using \reff{eq:phi *}).}
\end{align}
Since $M$ is separable and $|TM|\leq\beth_1$, we have $|C(M,TM)|\leq\beth_1$. Using $\Om(M)\sub C(M,TM)$, \reff{eq:Pi}, and Proposition \ref{prop:card MM}, it follows that
\[\left\vert\OM=\bigcup_{\M\in\MM}F^{-1}(\M)\right\vert\leq\beth_1^2=\beth_1.\]
This proves Corollary \ref{cor:card sec}.
\end{proof}
\begin{Rmk} \textnormal{Let $n\geq2$. Then the set of diffeomorphism types of manifolds of dimension $n$ has cardinality equal to $\beth_1$. To see this, we choose a countable set $\M$ of nondiffeomorphic connected $n$-manifolds. The map
\[\{0,1\}^\M\ni u\mapsto\Disj_{M\in\M:\,u(M)=1}M\in\{\textrm{n-\textrm{manifold}}\}\]
is injective. Hence the set of diffeomorphism types of manifolds of dimension $n$ has cardinality $\geq\beth_1$. Combining this with Proposition \ref{prop:card MM}, it follows that this cardinality equals $\beth_1$, as claimed.}
\end{Rmk}

\section{Proof of Theorem \ref{thm:gen ell} (monotone generation for ellipsoids)}\label{sec:gen ell}
Theorem \ref{thm:gen ell} follows from McDuff's characterization of the existence of symplectic embeddings between ellipsoids, and the fact that monotone generation is equivalent to almost order-reflexion. To explain this, let $(S,\leq)$ be a preordered set. We fix an order-preserving $(0,\infty)$-action on $S$. We define the \emph{order-capacity function} $c^\leq:S\x S\to[0,\infty]$ by
\[c^\leq(s,s'):=\sup\big\{a\in(0,\infty)\,\big|\,as\leq s'\big\}.\]
\begin{rmk}\label{rmk:c leq} \textnormal{For every $s\in S$ the function $c^\leq(s,\cdot)$ is a capacity, as defined in \reff{eq:caps S}.}
\end{rmk}
Let $\G\sub\caps(S)$. We call $\G$ \emph{almost order-reflecting} iff for all $s,s'\in S$ the following holds:
\[c(s)\leq c(s'),\,\forall c\in\G\then c^\leq(s,s')\geq1.\]
\begin{Rmk} \textnormal{A map $f$ between two preordered sets is called \emph{order-reflecting} if $f(s)\leq f(s')$ implies that $s\leq s'$. The set $\G$ is almost order-reflecting iff its evaluation map is ``almost'' order-reflecting, in the sense that $\ev_\G(s)\leq\ev_\G(s')$ implies that for every $a_0\in(0,1)$ there exists an $a\in[a_0,\infty)$, such that $as\leq s'$.
}
\end{Rmk}
\begin{prop}[characterization of monotone generation]\label{prop:mon gen} The set $\G$ monotonely generates if and only if it is almost order-reflecting.\end{prop}
In the proof of this result we use the following. Let $(X,\leq),(X',\leq')$ be preordered sets, $X_0\sub X$, and $f:X_0\to X'$. \label{mon} We define the \emph{monotonization of $f$} to be the map $F:X\to X'$ given by
\[F(x):=\sup\big\{f(x_0)\,\big|\,x_0\in X_0:\,x_0\leq x\big\}.\]
\begin{rmks}[monotonization]\label{rmk:mon}\textnormal{
\begin{enui}\item\label{rmk:mon:mon} The map $F$ is monotone.
\item\label{rmk:mon:homog} If $X$ and $X'$ are equipped with order-preserving $(0,\infty)$-actions and $f$ is homogeneous, then its monotonization is homogeneous.
\item\label{rmk:mon:restr} If $f$ is monotone then it agrees with the restriction of $F$ to $X_0$.
\end{enui}
}
\end{rmks}
\journal{
\begin{proof}[Proof of Proposition \ref{prop:mon gen}]\setcounter{claim}{0} This follows from an elementary argument. The implication ``$\follows$'' uses Remark \ref{rmk:mon}(\ref{rmk:mon:mon},\ref{rmk:mon:restr}). See the arXiv-version of this article.
\end{proof}
}
\arxiv{
\begin{proof}[Proof of Proposition \ref{prop:mon gen}]\setcounter{claim}{0} ``$\then$'': Assume that $\G$ monotonely generates. Let $s,s'\in S$ be such that $c(s)\leq c(s')$, for every $c\in\G$. This means that
\begin{equation}\label{eq:ev GG s}\ev_\G(s)\leq\ev_\G(s').
\end{equation}
By Remark \ref{rmk:c leq} and our assumption there exists a monotone function $F:[0,\infty]^\G\to[0,\infty]$, such that
\[c_s:=c^\leq(s,\cdot)=F\circ\ev_\G.\]
We have
\begin{align*}1&\leq c_s(s)\qquad\textrm{(since $\leq$ is reflexive and hence $s\leq s$)}\\
&=F\circ\ev_\G(s)\\
&\leq F\circ\ev_\G(s')\qquad\textrm{(using \reff{eq:ev GG s} and monotonicity of $F$)}\\
&=c_s(s').
\end{align*}
Hence $\G$ is almost order-reflecting. This proves ``$\then$''.\\

To prove the implication ``$\follows$'', assume that $\G$ is almost order-reflecting. Let $c_0\in\caps(S)$. 
\begin{claim}\label{claim:c s} For every pair of points $s,s'\in S$, satisfying $\ev_\G(s)\leq\ev_\G(s')$, we have $c_0(s)\leq c_0(s')$.
\end{claim}
\begin{proof} Since $c(s)\leq c(s')$, for every $c\in\G$, by assumption, we have $c_s(s')\geq1$. Let $a_0\in(0,1)$. It follows that there exists $a\in[a_0,\infty)$, such that $as\leq s'$. It follows that
\[a_0c_0(s)\leq ac_0(s)=c_0(as)\leq c_0(s').\]
Since this holds for every $a_0\in(0,1)$, it follows that $c_0(s)\leq c_0(s')$. This proves Claim \ref{claim:c s}.
\end{proof}
We define $f:\im(\ev_\G)\to[0,\infty]$ by setting $f(x):=c_0(s)$, where $s$ is an arbitrary point in $\ev_\G^{-1}(x)\sub S$. By Claim \ref{claim:c s} this function is well-defined, i.e., it does not depend on the choice of $s$. It satisfies
\begin{equation}\label{eq:f ev GG c}f\circ\ev_\G=c_0.\end{equation}
It follows from this equality and Claim \ref{claim:c s} that $f$ is monotone. By Remark \ref{rmk:mon}(\ref{rmk:mon:mon},\ref{rmk:mon:restr}) and equality \eqref{eq:f ev GG c} the monotonization $F$ of $f$ is a monotone function on $[0,\infty]^\G$ that satisfies $F\circ\ev_\G=c_0$. This proves ``$\follows$'' and completes the proof of Proposition \ref{prop:mon gen}.
\end{proof}
}

\begin{proof}[Proof of Theorem \ref{thm:gen ell}] We equip the set of ellipsoids in $(V,\om)$ with the preorder $E\leq E'$ iff there exists a symplectic embedding of $E$ into $E'$. By Theorem 1.1 in D.~McDuff's article \cite{McDHofer} the condition $c^V_j(E)\leq c^V_j(E')$, for all $j\in\N_0$, implies that $E$ symplectically embeds into $E'$. It follows that the set of all $c^V_j$ (with $j\in\N_0$) is almost order-reflecting. Hence by Proposition \ref{prop:mon gen} this set monotonely generates. This proves Theorem \ref{thm:gen ell}.
\end{proof}

\bibliographystyle{alpha}
\bibliography{References}

\end{document}